\documentclass[runningheads]{llncs}
\usepackage{graphicx}
\usepackage{amsfonts,amssymb,amstext,amsmath,color}
\usepackage{graphicx} 

\usepackage{mathrsfs}
\usepackage{comment}

\newcommand{\Mrm}{{\rm M}}

\newcommand{\Pb}{\mathbb{P}}
\newcommand{\GL}{{\rm GL}}

\newcommand{\Vcal}{\mathcal{V}}
\newcommand{\Hcal}{\mathcal{H}}

\newcommand{\Ubb}{\mathbb{U}}

\newcommand{\PC}{\mathscr{PC}}
\newcommand{\proE}{{\rm proE}}
\newcommand{\proHS}{{\rm proHS}}
\newcommand{\BW}{{\rm BW}}
\newcommand{\Wass}{{\rm W}}
\newcommand{\Xbf}{\mathbf{X}}
\newcommand{\Ybf}{\mathbf{Y}}

\newcommand{\mean}{\mathrm{mean}}
\newcommand{\cov}{\mathrm{cov}}
\newcommand{\mapto}{\ensuremath{\rightarrow}}

\newcommand{\approach}{\ensuremath{\rightarrow}}
\newcommand{\imply}{\ensuremath{\Rightarrow}}

\newcommand{\equivalent}{\ensuremath{\Longleftrightarrow}}

\newcommand{\X}{\mathcal{X}}

\newcommand{\N}{\mathbb{N}}
\newcommand{\R}{\mathbb{R}}

\newcommand{\bP}{\mathbb{P}}

\newcommand{\M}{\mathcal{M}}

\newcommand{\la}{\langle}
\newcommand{\ra}{\rangle}

\renewcommand{\H}{\mathcal{H}}
\renewcommand{\L}{\mathcal{L}}

\newcommand{\Mcal}{\mathcal{M}}
\newcommand{\Lcal}{\mathcal{L}}

\def\cov{{\rm cov}}

\def\trace{{\rm tr}}

\newcommand\HS{{\rm HS}}
\newcommand\eHS{{\rm HS_X}}

\def\Sym{{\rm Sym}}
\def\logHS{{\rm logHS}}

\def\logE{{\rm logE}}

\def\tr{{\rm tr}}

\def\Tr{{\rm Tr}}

\renewcommand{\b}{\mathbf{b}}

\def\1{\mathbf{1}}

\newcommand{\compose}{\circ}

\newcommand{\Ncal}{\mathcal{N}}

\begin{document}
\title{Alpha Procrustes metrics between positive definite operators: a unifying formulation for the Bures-Wasserstein and Log-Euclidean/Log-Hilbert-Schmidt metrics}
%
\titlerunning{Unifying Wasserstein and Log-Euclidean/Log-Hilbert-Schmidt metrics}
%
\author{H\`a Quang Minh\orcidID{0000-0003-3926-8875}}
\authorrunning{H\`a Quang Minh}
%
\institute{RIKEN Center for Advanced Intelligence Project, Tokyo, JAPAN
\\
\email{minh.haquang@riken.jp}}
%
\maketitle              
\begin{abstract}
This work presents a parametrized family of distances, namely the Alpha Procrustes distances, on the set of symmetric, positive definite 
(SPD) matrices. The Alpha Procrustes distances provide a unified formulation encompassing both the Bures-Wasserstein
and Log-Euclidean distances between SPD matrices. 
We show that the Alpha Procrustes distances are the Riemannian distances corresponding
to a family of Riemannian metrics
on the manifold of SPD matrices, which encompass both the Log-Euclidean and Wasserstein Riemannian metrics.
This formulation is then generalized to the set of positive definite Hilbert-Schmidt operators
on a Hilbert space, unifying the infinite-dimensional Bures-Wasserstein and Log-Hilbert-Schmidt distances.
In the setting of reproducing kernel Hilbert spaces (RKHS) covariance operators, we obtain
closed form formulas for all the distances via the corresponding kernel Gram matrices.
From a statistical viewpoint, the Alpha Procrustes distances give rise to a parametrized family of distances between Gaussian measures on Euclidean space, in the finite-dimensional case, and
separable Hilbert spaces, in the infinite-dimensional case, encompassing the 2-Wasserstein distance, with closed form formulas via Gram matrices in the RKHS setting.
The presented formulations are new both in the finite and infinite-dimensional settings.

\end{abstract}
\section{Introduction and motivations}

The main purpose of the current work is to provide a unified formulation linking two important distances on the set of 
symmetric, positive definite (SPD) matrices, namely the Bures-Wasserstein and Log-Euclidean distances, along 
with their infinite-dimensional generalizations on the set of positive definite Hilbert-Schmidt operators on an infinite-dimensional Hilbert space. 
As a consequence of this formulation, we also obtain a parametrized family of distances generalizing the $\L^2$-Wasserstein distance between  Gaussian measures on Euclidean and 
infinite-dimensional Hilbert spaces.

Let $\Sym^{+}(n)$ denote the set of $n \times n$ real, symmetric, positive semi-definite matrices and 
$\Sym^{++}(n) \subset \Sym^{+}(n) $ denote the set of
symmetric positive definite (SPD) matrices. Let $\Ubb(n)$ denote the set of $n \times n$ unitary 
matrices. 
In the context of optimal transport theory \cite{Villani:2008Optimal},
the Bures-Wasserstein distance on $\Sym^{+}(n)$
arises as follows. 
Let $\mu_X \sim \Ncal(m_1, A)$ and $\mu_Y \sim \Ncal(m_2, B)$ be two Gaussian probability distributions on $\R^n$. Let
$\Gamma(\mu_X, \mu_Y)$ be the set of all probability distributions on $\R^n \times \R^n$ whose marginal 
distributions are $\mu_X$ and $\mu_Y$. It was proved \cite{Downson:1982,Olkin:1982,Gelbrich:1990Wasserstein,Givens:1984} that the following is a squared distance,
the so-called $\L^2$-Wasserstein distance, between $\mu_X$ and $\mu_Y$
\begin{align}
\label{equation:L2-Wasserstein}
d^2_{\Wass}(\mu_X, \mu_Y) &= \inf_{\mu \in \Gamma(\mu_X, \mu_Y)}\int_{\R^n \times \R^n}||x-y||^2d\mu(x,y)
\nonumber
\\
&= ||m_1 - m_2||^2 + \trace[A+B -2 (A^{1/2}BA^{1/2})^{1/2}].
\end{align}
For $m_1 = m_2 = 0$, we obtain the Bures-Wasserstein distance on $\Sym^{+}(n)$
\begin{align}
\label{equation:Bures-finite}
d_{\BW}(A,B) = \left(\trace[A+B -2 (A^{1/2}BA^{1/2})^{1/2}]\right)^{1/2}.
\end{align}
From the viewpoint of Procrustes distances \cite{Bhatia:2018Bures}, \cite{Masarotto:2018Procrustes}, $d_{\BW}$ is obtained via the following optimization problem
\begin{align}
\label{equation:Bures-optimize-finite}
d_{\BW}(A,B) = \min_{U \in \Ubb(n)}||A^{1/2}U - B^{1/2}V||_F = \min_{U \in \Ubb(n)}||A^{1/2} - B^{1/2}U||_F,
\end{align}
where $||\;||_F$ is the Frobenius norm. Both the optimal transport and Procrustes distance formulations remain valid
in the infinite-dimensional settings, where $\mu_X, \mu_Y$ are two Gaussian measures on a Hilbert space $\H$,
$A,B$ are two covariance operators on $\H$ \cite{Gelbrich:1990Wasserstein,Masarotto:2018Procrustes}, with $||\;||_F$ replaced by the Hilbert-Schmidt norm $||\;||_{\HS}$.

\begin{remark}
Since the  eigenvalues $\{\lambda_k\}_{k=1}^n$ of $A^{1/2}BA^{1/2}$, which are all non-negative, are the same as those of $AB$, we can also write
\begin{align}
\label{equation:equivalent}
\trace[(A^{1/2}BA^{1/2})^{1/2}] = \sum_{k=1}^n\lambda_k^{1/2} = \trace(AB)^{1/2}.
\end{align}
This is the expression found in e.g. \cite{Downson:1982}.
Subsequently, we also make use of the expression $\trace(AB)^{1/2}$ in the sense of Eq.(\ref{equation:equivalent}).
If $A \in \Sym^{++}(n)$ (or $B \in \Sym^{++}(n))$, then $(AB)^{1/2}$ is also well-defined analytically and is given by
\begin{align}
(AB)^{1/2} = A^{1/2}(A^{1/2}BA^{1/2})^{1/2}A^{-1/2}. 
\end{align}
\end{remark}

The Log-Euclidean distance, on the other hand, is the Riemannian distance associated with the bi-invariant Riemannian metric
on $\Sym^{++}(n)$ \cite{LogEuclidean:SIAM2007}, considered as a Lie group under the commutative multiplication
$A \odot B = \exp(\log(A) + \log(B))$, where $\log$ denotes the principal matrix logarithm. It is given by 
\begin{align}
\label{equation:logE}
d_{\logE}(A,B) = ||\log(A) - \log(B)||_F.
\end{align}

{\bf Contributions of this work}. While the two distances given in Eqs.(\ref{equation:Bures-finite}) and (\ref{equation:logE}) appear quite different and unrelated, we show that
\begin{enumerate}
	\item By generalizing the Procrustes distance optimization problem in Eq.(\ref{equation:Bures-optimize-finite}), 
	we obtain a parametrized family of distances on $\Sym^{++}(n)$ that includes both
	the Bures-Wasserstein and Log-Euclidean distances as special cases. We call this family {\it Alpha Procrustes distances}.
	
	\item We show that the Alpha Procrustes distances are in fact the Riemannian distances
	corresponding to a family of Riemannian metrics on $\Sym^{++}(n)$, which include
	the Log-Euclidean Riemannian metric in \cite{LogEuclidean:SIAM2007} and the Wasserstein Riemannian metric \cite{Takatsu2011wasserstein,Bhatia:2018Bures,Malago:WassersteinGaussian2018}
	in as special cases.

	\item The Alpha Procrustes distances are then generalized from $\Sym^{++}(n)$ to the set
	of positive definite unitized Hilbert-Schmidt operators on an infinite-dimensional Hilbert space $\H$.
	This setting is more general than the setting of covariance operators on $\H$, which form a strict subset of
	the set of Hilbert-Schmidt operators on $\H$ when $\dim(\H) = \infty$.
	In particular, we recover the infinite-dimensional Bures-Wasserstein and Log-Hilbert-Schmidt distances
	\cite{MinhSB:NIPS2014} as special cases. In the setting of reproducing kernel Hilbert spaces (RKHS), we obtain closed form formulas for
	all the distances between RKHS covariance operators, via the corresponding kernel Gram matrices.
	
	\item As a statistical consequence of the Alpha Procrustes distance formulation, we obtain a parametrized family of distances between Gaussian measures
	on Euclidean and Hilbert spaces that include the $\L^2$-Wasserstein distance as a special case. 
	In the RKHS setting, we obtain closed form formulas for all the distances, via the corresponding kernel Gram matrices.
	
\end{enumerate}

An extended abstract containing some of the main results in this paper, without proofs, 
has appeared in the proceedings of the conference Geometric Science of Information 2019 
\cite{Minh:GSI2019}.

\section{Finite-dimensional setting}
\label{section:finite-dim}
	
We start with the sets $\Sym^{+}(n)$ and $\Sym^{++}(n)$ of $n \times n$ real positive semi-definite and positive definite matrices, respectively. The Procrustes distance formulation in
Eq.(\ref{equation:Bures-optimize-finite}) can be generalized as follows.

\begin{definition}
	[\textbf{Alpha Procrustes distance - finite-dimensional version}]
	\label{definition:alpha-procrustes-finite}
	Let $\alpha \in \R, \alpha \neq 0$ be fixed.
	The $\alpha$-Procrustes distance between two matrices $A, B \in \Sym^{++}(n)$  is defined to be
	\begin{align}
	d^{\alpha}_{\proE}(A,B) = \min_{U,V \in \Ubb(n)}\left\|\frac{A^{\alpha}U - B^{\alpha}V}{\alpha}\right\|_F = \min_{U \in \Ubb(n)}\left\|\frac{A^{\alpha} - B^{\alpha}U}{\alpha}\right\|_F.
	\end{align}
	For $\alpha > 0$, we define this distance over the larger set $\Sym^{+}(n)$.
\end{definition}


\begin{theorem}
	[\textbf{Explicit expression}]
	Let either (i) $A, B \in \Sym^{++}(n)$, $\alpha \in \R, \alpha \neq 0$, or (ii) $A, B \in \Sym^{+}(n), \alpha \in \R, \alpha > 0$. Then
	\begin{align}
	\label{equation:dproE-finite}
	d^{\alpha}_{\proE}(A,B)
	&= \left(\frac{1}{\alpha^2}\trace(A^{2\alpha} + B^{2\alpha} - 2\trace(A^{\alpha}B^{2\alpha}A^{\alpha})^{1/2})\right)^{1/2}.
	\end{align}
\end{theorem}

{\bf Special case: Bures-Wasserstein-Fr\'echet distance}. For $\alpha = 1/2$, $A,B \in \Sym^{+}(n)$, we obtain
\begin{align}
d_{\proE}^{1/2}(A,B) = 2(\trace[A + B - (A^{1/2}BA^{1/2})^{1/2}])^{1/2} =2 d_{\BW}(A,B).
\end{align}
This is precisely twice the Bures-Wasserstein-Fr\'echet distance \cite{Downson:1982,Olkin:1982,Malago:WassersteinGaussian2018,Bhatia:2018Bures}.

{\bf Special case: $A,B$ commute}. In this case, Eq.(\ref{equation:dproE-finite}) becomes
\begin{align}
d^{\alpha}_{\proE}(A,B) = \left\|\frac{A^{\alpha} - B^{\alpha}}{\alpha}\right\|_F.
\end{align}
This is precisely the power Euclidean distance \cite{Dryden:2009}. For $A,B \in \Sym^{++}(n)$, 
\begin{align}
\lim_{\alpha \approach 0}\left\|\frac{A^{\alpha} - B^{\alpha}}{\alpha}\right\|_F = ||\log(A) - \log(B)||_F = d_{\logE}(A,B).
\end{align}
The following
 shows that this limit also holds
for the Alpha Procrustes distance.

\begin{theorem}
	[\textbf{Limiting case - Log-Euclidean distance}]
	\label{theorem:limit-alpha-0-finite}
	Let $A,B \in \Sym^{++}(n)$ be fixed. Then
	\begin{align}
	\label{equation:limit-alpha-0-finite}
	\lim_{\alpha \approach 0}\frac{1}{\alpha^2} [\trace(A^{2\alpha} +  B^{2\alpha} - 2(A^{\alpha}B^{2\alpha}A^{\alpha})^{1/2}] = ||\log(A) - \log(B)||^2_F.
	\end{align}
\end{theorem}

We have then the following result.
\begin{theorem}
	[\textbf{Parametrized family of distances - SPD matrices}]
	\label{theorem:metric-family-finite}
The function $d^{\alpha}_{\proE}$, as defined in Eq.(\ref{equation:dproE-finite}), is a metric on the set $\Sym^{++}(n)$ for all $\alpha \in \R$, with twice the Bures-Wasserstein-Fr\'echet distance corresponding to $\alpha = 1/2$ and the Log-Euclidean distance corresponding to $\alpha = 0$. Furthermore, $d^{\alpha}_{\proE}$ is a metric on the set $\Sym^{+}(n)$ for all $\alpha > 0$.
\end{theorem}

{\bf Comparison between the Alpha Procrustes and Power Euclidean distances}.
We observe from their corresponding explicit formulas that the Alpha Procrustes and Power Euclidean distances,
with the same power $\alpha$, coincide when $A$ and $B$ commute. We now show that
the Alpha Procrustes distance is {\it strictly smaller} than the Power Euclidean distance when $A$ and $B$ do not commute.
We make use of the Araki-Lieb-Thirring inequality \cite{Lieb1976inequalities,Araki1990inequality,Wang1995trace}. 
For two matrices $A,B \in \Sym^{+}(n)$ ($A,B \in \Sym^{++}(n)$ if $r < 0$), 
\begin{align}
\trace(B^{1/2}AB^{1/2})^{r} \geq \trace(A^{r}B^{r}), \;\;\; |r| \leq 1,
\end{align}
with equality if and only if $r = 0, 1 -1$, or $AB = BA$. In particular, for $r = 1/2$, 
\begin{align}
\label{equation:Araki-Lieb-Thirring}
\trace[(A^{1/2}BA^{1/2})^{1/2}] \geq \trace(A^{1/2}B^{1/2}),
\end{align}
with equality if and only if $AB = BA$. The following is a consequence of this inequality in our setting.

\begin{theorem} 
	[\textbf{Comparison between Alpha Procrustes and Power Euclidean distances}]	
	\label{theorem:power-euclidean-procrustes-compare}
	Let $\alpha \in \R, \alpha \neq 0$ be fixed but arbitrary. Then for any pair $A,B \in \Sym^{+}(n)$ ($A,B \in \Sym^{++}(n)$ if $\alpha < 0$),
	\begin{align}
	d_{\proE}^{\alpha}(A,B) &= \frac{(\trace[A^{2\alpha} + B^{2\alpha} - 2(A^{\alpha}B^{2\alpha}A^{\alpha})^{1/2}])^{1/2}}{|\alpha|}
	\nonumber
	\\
	&\leq d_{E,\alpha}(A,B) =\left\|\frac{A^{\alpha} - B^{\alpha}}{\alpha}\right\|_F,
	\end{align}
	with equality if and only if $A,B$ commute. For $A,B \in \Sym^{++}(n)$, at the limit $\alpha = 0$,
	both sides are equal to the Log-Euclidean distance $||\log(A)-\log(B)||_F$.
\end{theorem}

{\bf A parametrized family of distances between Gaussian measures on Euclidean space}. 
Each distance/divergence on the set $\Sym^{+}(n)$ corresponds to a distance/divergence on 
the set of zero-mean Gaussian measures $\Ncal(0,C)$, $C \in \Sym^{+}(n)$, and vice versa. Similarly,
each distance/divergence on the set $\Sym^{++}(n)$ corresponds to a distance/divergence 
on the set of zero-mean {\it non-degenerate} Gaussian measures $\Ncal(0,C)$, $C \in \Sym^{++}(n)$, and vice versa.

For general Gaussian measures of the form $\Ncal(m, C)$, $m \in \R^n$,
motivated by Eq.~(\ref{equation:L2-Wasserstein}) and Theorem \ref{theorem:metric-family-finite}, we arrive at the following result.
\begin{theorem}
	[\textbf{Parametrized family of distances - Gaussian measures on Euclidean space}]
	\label{theorem:metric-Gaussian-Rn}
	Let $d_{\mean}:\R^n \times \R^n\mapto \R^{+}$ be a metric on $\R^n$.
	The following is a parametrized family of squared distances on the set of non-degenerate Gaussian measures on $\R^n$,
	with $\alpha \in \R$, $\alpha \neq 0$,
	\begin{align}
	(D^{\alpha}_{\proE}[\Ncal(m_1,C_1)|| \Ncal(m_2,C_2)])^2 &= d^2_{\mean}(m_1,m_2)
	\nonumber
	\\
	&+ \frac{1}{4\alpha^2}\trace[C_1^{2\alpha} + C_2^{2\alpha} - 2(C_1^{\alpha}C_2^{2\alpha}C_1^{\alpha})^{1/2}].
	\end{align}
	The statememt remains true on the larger set of general Gaussian measures on $\R^n$ if we restrict $\alpha$ to $\alpha > 0$.
\end{theorem}

{\bf Special case: $\L^2$-Wasserstein distance}. For $d_{\mean}(m_1, m_2) = ||m_1 - m_2||_2$, the $\ell^2$-norm on $\R^n$ and $\alpha = 1/2$, we recover the squared $\L^2$-Wasserstein distance in Eq.~(\ref{equation:L2-Wasserstein}).

{\bf Limiting case}. As $\alpha \approach 0$, we obtain the following squared distance
\begin{align}
	&(D^{\alpha}_{\logE}[\Ncal(m_1,C_1)|| \Ncal(m_2,C_2)])^2 =\lim_{\alpha \approach 0}(D^{\alpha}_{\proE}[\Ncal(m_1,C_1)|| \Ncal(m_2,C_2)])^2
	\nonumber
	\\
	& = d^2_{\mean}(m_1,m_2) + \frac{1}{4}||\log(C_1) - \log(C_2)||^2_{F}.
\end{align}

\subsection{Riemannian geometry of Alpha Procrustes distances}
\label{section:Riemannian}

In this section, we present the Riemannian metric structures corresponding to the Alpha Procrustes distances. 
The method used to obtain the Riemannian metric and geodesics is a generalization of those in
the setting of the Wasserstein Riemannian metric \cite{Takatsu2011wasserstein},\cite{Bhatia:2018Bures}, \cite{Malago:WassersteinGaussian2018}.

We recall that a differentiable map $\pi: (\M, g) \mapto (\Ncal,h)$ between two Riemannian manifolds
$(\Mcal, g)$ and $(\Ncal, h)$ is said to be a smooth submersion if its differential map $D\pi(P): T_P\Mcal \mapto T_{\pi(P)}\Ncal$ is surjective
$\forall P \in \Mcal$. Consider the direct sum decomposition
\begin{align}
T_P\Mcal = \Vcal_P \oplus \Hcal_P = \ker(D\pi(P)) \oplus (\ker(D\pi(P)))^{\perp},
\end{align}
where $\Vcal_P$ and $\Hcal_P$ are the vertical and horizontal spaces, respectively, at $P$.
A smooth submersion $\pi$ is said to be a Riemannian submersion if the differential map
$D\pi(P):\Hcal_P \mapto T_{\pi(P)}\Ncal$ is isometric $\forall P \in \Mcal$.

Let $\Mrm(n)$ denote the vector space of $n \times n$ real matrices, equipped with the Frobeninus inner product $\la \;,\;\ra_F$, and $\GL(n) \subset \Mrm(n)$ the open set of invertible matrices.
Then $(\GL(n), \la \; ,\ra_F)$ is a Riemannian manifold, with the tangent space at each point identified with $\Mrm(n)$, along with the inherited Frobenius inner product.
We now proceed by constructing a smooth submersion from the Riemannian manifold
$(\GL(n), \la \; , \;\ra_F)$ onto $\Sym^{++}(n)$ and deriving the corresponding Riemannian metric
on $\Sym^{++}(n)$ so that this is a Riemannian submersion.

To obtain the subsequent results, we make use of the following properties of the 
exponential map $\exp: \Sym(n) \mapto \Sym^{++}(n)$, $\exp(A) = \sum_{k=0}^{\infty}\frac{A^k}{k!}$, and 
its inverse map $\log:\Sym^{++}(n) \mapto \Sym(n)$.
Since $\exp(\log(A_0)) = A_0$ for all $A_0 \in \Sym^{++}(n)$, by the chain rule
\begin{align}
D\exp(\log(A_0))\compose D\log(A_0)A = A, \;\;\;A \in \Sym(n).
\end{align}
Similarly, from $\log(\exp(B_0)) = B_0$ for all $B_0 \in \Sym(n)$, we obtain
\begin{align}
D\log(A_0) \compose D\exp(\log(A_0))A = A, \;\;A_0 = \exp(B_0), A \in \Sym(n).
\end{align} 
It thus follows that as linear operators on $\Sym(n)$, $D\log(A_0)$ and $D\exp(\log(A_0))$ are invertible and
\begin{align}
D\log(A_0) = [D\exp(\log(A_0))]^{-1}, \;\;\; A_0 \in \Sym^{++}(n).
\end{align}

Let $\alpha \in \R, \alpha \neq 0$ be fixed.
Consider the following map $\pi_{\alpha}: \GL(n) \mapto \Sym^{++}(n)$ defined by 
\begin{align}
\label{equation:submersion}
\pi_{\alpha}(A) = (\alpha^2 AA^{*})^{1/(2\alpha)} = \exp\left(\frac{1}{2\alpha}\log(g(A))\right),
\end{align}
where $g:\GL(n) \mapto \Sym^{++}(n)$ is defined by $g(A) = \alpha^2 AA^{*}$. 

\begin{proposition}
	\label{proposition:differential}
For each $\alpha \in \R, \alpha \neq 0$ fixed, the map $\pi_{\alpha}:\GL(n) \mapto \Sym^{++}(n)$, as defined in Eq.(\ref{equation:submersion}), is a smooth submersion.
	Let $A_0 \in \GL(n)$ be fixed. The differential map $D\pi_{\alpha}(A_0): \Mrm(n) \mapto \Sym(n)$ is given by
	\begin{align}
	D\pi_{\alpha}(A_0)(X) =& \frac{\alpha}{2}D\exp\left(\frac{1}{2\alpha}\log(\alpha^2 A_0A_0^{*})\right) \compose D\log(\alpha^2 A_0A_0^{*})(XA_0^{*} + A_0^{*}X),
	\nonumber
	\\
	&X \in \Mrm(n).
	\end{align}
	The vertical and horizontal spaces of the map $\pi_{\alpha}$ at $A_0$ are given by
	\begin{align}
	&\Vcal_{A_0} = \ker(D\pi_{\alpha}(A_0)) = (\Sym(n))^{\perp}(A_0^{*})^{-1},
	\\
	&\Hcal_{A_0} =  (\ker(D\pi_{\alpha}(A_0)))^{\perp} = \Sym(n)A_0.
	\end{align}
\end{proposition}

\begin{proposition}
	\label{proposition:general-lyapunov}
	Let $P_0 \in \Sym^{++}(n)$ be fixed. Let $\alpha \in \R$ be fixed.
	For any matrix $Y \in \Sym(n)$, there is a unique matrix $H \in \Sym(n)$ such that
	\begin{align} 
	\label{equation:general-lyapunov}
	& D\exp(\log(P_0)) \compose D\log(P_0^{2\alpha})(HP_0^{2\alpha} + P_0^{2\alpha} H) = Y.
	\end{align}
\end{proposition}
{\bf Notation}. We denote the unique matrix $H$ satisfying Eq.(\ref{equation:general-lyapunov}) by
\begin{align}
\label{equation:lyapunov-notation}
H = \Lcal_{P_0, \alpha}(Y).
\end{align}
{\bf Special case: Lyapunov equation}. For $\alpha = 1/2$ (the Wasserstein case), since
$D\exp(\log(P_0)) \compose D\log(P_0) = \1_{\Sym(n)}$, Eq.(\ref{equation:general-lyapunov}) reduces to
the Lyapunov equation
\begin{align}
HP_0 + P_0H = Y,
\end{align}
which has a unique solution $H \in \Sym(n)$.

We now construct the Riemannian metric on $\Sym^{++}(n)$ that turns $\pi_{\alpha}$ into a Riemannian submersion.
The following is the main result in this section.
\begin{theorem}
	[\textbf{Riemannian metric}]
	\label{theorem:general-metric}
	Define the following Riemannian metric on $\Sym^{++}(n)$. For each $P_0 \in \Sym^{++}(n)$, the 
	inner product on the tangent space $T_{P_0}(\Sym^{++}(n)) = \Sym(n)$ is given by
	\begin{align}
	\label{equation:general-metric}
	\la Y, Z\ra_{P_0} = 4\trace(\Lcal_{P_0, \alpha}(Y)P_0^{2\alpha}\Lcal_{P_0, \alpha}(Z)), \;\;\; Y, Z \in \Sym(n),
	\end{align}
	where $\Lcal_{P_0, \alpha}(Y)$ is as defined in Proposition \ref{proposition:general-lyapunov} and Eq.(\ref{equation:lyapunov-notation}). Then for any $\alpha \in \R, \alpha \neq 0$, the map
	$\pi_{\alpha}: (\GL(n), \la \; , \;\ra_F) \mapto (\Sym^{++}(n), \la \;, \;\ra_{P_0})$, as defined 
	in Eq.(\ref{equation:submersion}) is a Riemannian submersion.
\end{theorem}

{\bf Special case: Log-Euclidean metric}. In the limiting case $\alpha = 0$, Eq.(\ref{equation:general-lyapunov}) becomes
\begin{align}
D\exp(\log(P_0))(2H) = Y \equivalent H = \Lcal_{P_0, 0}(Y)& = \frac{1}{2}[D\exp(\log(P_0))]^{-1}Y
\nonumber
\\
& = \frac{1}{2}D\log(P_0)Y.
\end{align}
It follows that for any pair $Y,Z \in \Sym(n)$, we have
\begin{align}
\la Y, Z\ra_{P_0} = \la D\log(P_0)Y, D\log(P_0)W\ra_F.
\end{align}
This is precisely the Log-Euclidean Riemannian metric \cite{LogEuclidean:SIAM2007}.

{\bf Special case: Wasserstein Riemannian metric}. For $\alpha = 1/2$, since $D\exp(\log(P_0) \compose D\log(P_0) = \1_{\Sym(n)}$, Eq.(\ref{equation:general-lyapunov}) reduces to the Lyapunov equation
\begin{align}
HP_0 + P_0H  = Y.
\end{align}
Thus the metric in Eq.(\ref{equation:general-metric}) becomes four times the Wasserstein Riemannian metric given in \cite{Takatsu2011wasserstein}, \cite{Bhatia:2018Bures}, \cite{Malago:WassersteinGaussian2018}.

To obtain the geodesics on $(\Sym^{++}(n), \la \; , \;\ra_{P_0})$, we apply the following result.
\begin{theorem}
	\label{theorem:geodesic-submersion}
	Let $\pi: (\Mcal, g) \mapto (\Ncal, h)$ be a Riemannian submersion between two Riemannian manifolds
	$(\Mcal, g)$, $(\Ncal, h)$. Let $\gamma$ be a geodesic in $(\Mcal, g)$ with $\gamma{'}(0)$ being horizontal.
	Then
	\begin{enumerate}
		\item $\gamma{'}(t)$ is horizontal for all $t$.
		\item $\pi \compose \gamma$ is a geodesic in $(\Ncal, h)$ with the same length as $\gamma$.
	\end{enumerate}
	
\end{theorem}
\begin{theorem}
	[\textbf{Geodesic and Riemannian distance}]
	\label{theorem:geodesic}
	Let $\alpha \in \R, \alpha \neq 0$ be fixed.
	Under the Riemannian metric on $\Sym^{++}(n)$ defined in Theorem \ref{theorem:general-metric},
	the Riemannian distance between $A,B \in \Sym^{++}(n)$ is precisely the Alpha Procrustes distance, namely
	\begin{align}
	d(A,B) = \frac{1}{|\alpha|}(\trace[A^{2\alpha} + B^{2\alpha} - 2(A^{\alpha}B^{2\alpha}A^{\alpha})^{1/2}])^{1/2}.
	\end{align}
	Furthermore, the distance $d(A,B)$ is the length of the following geodesic
	\begin{align}
	\gamma(t) =& [(1-t)^2A^{2\alpha} + t^2B^{2\alpha} + t(1-t)[(A^{2\alpha}B^{2\alpha})^{1/2} + (B^{2\alpha}A^{2\alpha})^{1/2}]]^{1/(2\alpha)},
	\nonumber
	\\
	&\gamma(0) = A, \gamma(1) = B.
	\end{align}
\end{theorem}

\section{Infinite-dimensional setting}
\label{section:infinite-dim}

We now generalize the results in Section \ref{section:finite-dim} to the infinite-dimensional setting.
Throughout the following, let $\H$ denote a real, separable, infinite-dimensional Hilbert space, unless
explicitly stated otherwise. Let $\L(\H)$ denote the set of bounded linear operators on $\H$, $\Sym(\H) \subset \L(\H)$
the set of bounded, self-adjoint operators, $\Sym^{+} \subset \Sym(\H)$ and $\Sym^{++}(\H) \subset \Sym^{+}(\H)$ denote, respectively,
the sets of positive and strictly positive operators on $\H$. 
Let $\Ubb(\H)$ denote the set of unitary operators on $\H$ and $\Tr(\H)$ and $\HS(\H)$ denote the sets of
trace class and Hilbert-Schmidt operators on $\H$, respectively.

In the case $A,B$ are two positive trace class operators
on $\H$,
we have
\begin{theorem}
	[\cite{Gelbrich:1990Wasserstein,Masarotto:2018Procrustes}]
	\label{theorem:Bures-traceclass}
	Let $A,B \in \Sym^{+}(\H) \cap \Tr(\H)$ be fixed. Then
	\begin{align}
	\min_{U \in \Ubb(\H)}||A^{1/2} - B^{1/2}U||^2_{\HS} &= \trace[A + B - 2(A^{1/2}BA^{1/2})^{1/2}] 
	\\
	&=d^2_{\Wass}(\Ncal(0,A),\Ncal(0,B)).
	\end{align}
\end{theorem}

\begin{corollary}
	\label{corollary:Bures-alpha}
	Let $\alpha \in \R, \alpha > 0$ be fixed. Let $A,B \in \Sym^{+}(\H)$ be fixed, such that $A^{\alpha}, B^{\alpha} \in \HS(\H)$. Then
	\begin{align}
	\label{equation:Bures-alpha}
	\min_{U \in \Ubb(\H)}||A^{\alpha} - B^{\alpha}U||^2_{\HS} = \trace[A^{2\alpha} + B^{2\alpha} - 2(A^{\alpha}B^{2\alpha}A^{\alpha})^{1/2}].
	\end{align}
\end{corollary}

While
Eq.(\ref{equation:Bures-alpha}) is valid for any pair $A, B \in \Sym^{+}(\H)$
such that $A^{\alpha}, B^{\alpha} \in \HS(\H)$, in general the limit 
$\lim_{\alpha \approach 0}\frac{1}{\alpha} \min_{U \in \Ubb(\H)}||A^{\alpha} - B^{\alpha}U||_{\HS}$
does not exist in a form similar to Eq.(\ref{equation:limit-alpha-0-finite}), since $\log(A)$ is unbounded even when $A$ is strictly positive.
To obtain the infinite-dimensional generalization of Theorem
\ref{theorem:limit-alpha-0-finite}, we consider the setting
of positive definite unitized Hilbert-Schmidt operators. Let us denote by $\bP(\H)$ the set
of self-adjoint, positive definite operators on $\H$
$\bP(\H) = \{A \in \L(\H), A^{*} = A, \exists M_A > 0 \; s.t. \la x, Ax\ra \geq M_A||x||^2 \; \forall x \in \H\}$.
We write $A > 0 \equivalent A \in \bP(\H)$.
We recall that in \cite{Larotonda:2007}, the author defined the set of {\it extended (or unitized) Hilbert-Schmidt operators} 
on $\H$ to be
\begin{align}
\HS_X(\H) = \{ A + \gamma I : A \in \HS(\H), \gamma \in \R\},
\end{align}
which becomes a Hilbert space under the {\it extended Hilbert-Schmidt inner product}
\begin{align}
\la (A+\gamma I), (B+ \nu I)\ra_{\eHS} = \la A,B\ra_{\HS} +\gamma \nu = \trace(A^{*}B) + \gamma\nu.
\end{align}
The set of positive definite unitized (or extended) Hilbert-Schmidt operators,
which is an infinite-dimensional generalization
of $\Sym^{++}(n)$,
 is defined to be
\begin{align}
\PC_2(\H) = \Pb(\H) \cap \HS_X(\H) = \{A+\gamma I > 0: A \in \HS(\H), \gamma \in \R\}.
\end{align}
Similarly, in \cite{Minh:LogDet2016}, the set of {\it extended (or unitized) trace class operators} is defined to be
\begin{align}
\Tr_X(\H) = \{ A + \gamma I: A \in \Tr(\H), \gamma \in \R\},
\end{align}
and the set of positive definite unitized (or extended) trace class operators is defined to be
\begin{align}
\PC_1(\H) = \Pb(\H) \cap \Tr_X(\H) = \{A+\gamma I > 0: A \in \Tr(\H), \gamma \in \R\}.
\end{align}
In particular, if $A$ is a covariance operator on $\H$, then $(A+\gamma I) \in \PC_1(\H)$ for any $\gamma \in \R, \gamma > 0$.

The set $\PC_2(\H)$ is a Hilbert manifold on which one can define the infinite-dimensional generalizations of the affine-invariant Riemannian metric \cite{Larotonda:2007}
and 
the Log-Determinant divergences \cite{Minh:LogDet2016,Minh:LogDetIII2018}.
In \cite{MinhSB:NIPS2014}, we introduced the Log-Hilbert-Schmidt distance on $\PC_2(\H)$,
which generalizes the Log-Euclidean distance on $\Sym^{++}(n)$. 
For $(A+\gamma I), (B +\nu I) \in \PC_2(\H)$, 
this distance is defined by
\begin{align}
d_{\logHS}[(A+\gamma I), (B+ \nu I)] = ||\log(A+\gamma I) - \log(B+\nu I)||_{\eHS}.
\end{align}

We next define a family of distances that includes both the infinite-dimensional Bures-Wasserstein
and Log-Hilbert-Schmidt distances as special cases.

\subsection{The case $\gamma = \nu =1$}
\label{section:gamma-nu-1}

For a fixed $\gamma  \in \R, \gamma > 0$, we consider the following subset of $\PC_2(\H)$
\begin{align}
\PC_2(\H)(\gamma) = \{ A+\gamma I > 0 : A \in \HS\} \subset \PC_2(\H).
\end{align}
We first generalize the Alpha Procrustes distance in Definition \ref{definition:alpha-procrustes-finite}
to the set 
\begin{align}
\PC_2(\H)(1) = \{ I+A > 0 : A \in \HS\} \subset \PC_2(\H).
\end{align}
We first note that for any $(I+A) \in \PC_2(\H)(1)$, 
$(I+A)^{\alpha} = \exp[\alpha\log(I+A)] = I + \sum_{k=1}^{\infty}\frac{\alpha^k}{k!}[\log(I+A)]^k  = I+C$
where $C \in \Sym(\H) \cap \HS(\H)$, since
$||C||_{\HS} \leq \sum_{k=1}^{\infty}\frac{|\alpha|^k}{k!}||\log(I+A)||_{\HS}^k  = \exp(||\log(I+A)||_{\HS}) - 1 < \infty$.

\begin{proposition}
	\label{proposition:limit-UV-same}
	Let $(I+A), (I+B) \in \PC_2(\H)(1)$ and $\alpha \in \R$ be fixed. Then 	
\begin{align}
&\min_{(I+U), (I+V) \in \Ubb(\H) \cap \HS_X(\H)}||(I+A)^{\alpha}(I+U) - (I+B)^{\alpha}(I+V)||_{\HS}
\label{equation:limit-UV-same}
\\
= &\min_{(I+V) \in \Ubb(\H) \cap \HS_X(\H)}||(I+A)^{\alpha} - (I+B)^{\alpha}(I+V)||_{\HS}.
\label{equation:limit-V-same}
\end{align}
\end{proposition}

Let us briefly motivate the operators  of the form
$(I+U) \in \Ubb(\H) \cap \HS_X(\H)$ in Proposition \ref{proposition:limit-UV-same}. We first note
that if $U \in \Ubb(\H)$ and $(I+A) \in \HS_X(\H)$, then we generally do not have
$(I+A)U = U + AU\in \HS_X(\H)$. If we consider operators of the form $(I+U) \in \Ubb(\H) \cap \HS_X(\H)$, then
$(I+A)(I+U) = I + A + U + AU \in \HS_X(\H)$ since $U \in \HS(\H)$.
Furthermore, for $(I+A), (I+B), (I+U), (I+V)\in \HS_X(\H)$, we have
$(I+A)(I+U) - (I+B)(I+V) \in \HS(\H)$,
so that the expressions in Eqs.(\ref{equation:limit-UV-same}) and (\ref{equation:limit-V-same}) are both well-defined and finite.
The form $I+U \in \Ubb(\H) \cap \HS_X(\H)$ is also motivated by the following
polar decomposition.

\begin{lemma}
	\label{lemma:polar-decomp-HSX}
	Let $(A+\gamma I) \in \HS_X(\H)$ be invertible. Then its polar decomposition has the form
	\begin{align}
	A+\gamma I = S|A+\gamma I| = (I+R)|A+\gamma I|,
	\end{align}
	where $S = I+R \in \Ubb(\H) \cap \HS_X(\H)$.
	If, furthermore, $(A+\gamma I) \in \Tr_X(\H)$, then $S = (I+R) \in \Ubb(\H) \cap \Tr_X(\H)$.
\end{lemma}

Motivated by Proposition \ref{proposition:limit-UV-same}, the following is our definition
for the Alpha Procrustes distance between operators of the form $(I+A)> 0$, $A \in \HS(\H)$.

\begin{definition}
	[\textbf{Alpha Procrustes distance between positive definite Hilbert-Schmidt operators - special case}]
	\label{definition:alpha-procrustes-infinite-1}
	Let $\alpha \in \R, \alpha \neq 0$ be fixed.
The $\alpha$-Procrustes distance 
between
$(I+A), (I+B) \in \PC_2(\H)(1)$ is defined to be	
\begin{align}
\label{equation:alpha-procrustes-infinite-1}
d^{\alpha}_{\proHS}[(I+A), (I+B)]= \min_{(I+U) \in \Ubb(\H) \cap \HS_X(\H)}\left\|\frac{(I+A)^{\alpha} - (I+B)^{\alpha}(I+U)}{\alpha}\right\|_{\HS}.
\end{align}
\end{definition}

To state the explicit expressions for $d^{\alpha}_{\proHS}[(I+A), (I+B)]$, as defined in Eq.(\ref{equation:alpha-procrustes-infinite-1}), we first need the following result.


\begin{proposition}
	\label{proposition:HS-to-Tr-alpha}
	Let $(I+A), (I+B) \in \PC_2(\H)$. Let $\alpha \in \R$ be fixed. Then
	\begin{align}
	(I+A)^{2\alpha} + (I+B)^{2\alpha} - 2[(I+A)^{\alpha}(I+B)^{2\alpha}(I+A)^{\alpha}]^{1/2} \in \Tr(\H),
	\\
	[(I+A)^{\alpha}(I+B)^{2\alpha}(I+A)^{\alpha}]^{1/2} - (I+A)^{\alpha} - (I+B)^{\alpha} + I \in \Tr(\H).
	\end{align}
\end{proposition}

With Proposition \ref{proposition:HS-to-Tr-alpha},
 we are ready to state the following.

\begin{theorem}
	\label{theorem:Pdistance-AI-HS}
	Let $(I+A), (I+B) \in \PC_2(\H)$ and $\alpha \in \R, \alpha \neq 0$ be fixed.
	Then
	\begin{align}
	&(d^{\alpha}_{\proHS}[(I+A), (I+B)])^2= 
	\frac{1}{\alpha^2}\min_{(I+U) \in \Ubb(\H) \cap \HS_X(\H)}||(I+A)^{\alpha} - (I+B)^{\alpha}(I+U)||^2_{\eHS} 
	\nonumber
	\\
	& = \frac{1}{\alpha^2}\left(\trace[(I+A)^{2\alpha} + (I+B)^{2\alpha} -2 [(I+A)^{\alpha}(I+B)^{2\alpha}(I+A)^{\alpha}]^{1/2}]\right)
	\label{equation:min-AI-HS-1}
	\\
	& = \frac{1}{\alpha^2}\left(||(I+A)^{\alpha}||^2_{\eHS} + ||(I+B)^{\alpha}||^2_{\eHS} - 2\right)
	\nonumber
	\\
	& - \frac{2}{\alpha^2}\left(\trace[[(I+A)^{\alpha}(I+B)^{2\alpha}(I+A)^{\alpha}]^{1/2} - (I+A)^{\alpha} - (I+B)^{\alpha} + I]\right).
	\label{equation:min-AI-HS-2}
	\end{align}
	In particular, for $(I+A), (I+B) \in \PC_1(\H)$,
	\begin{align}
	(d^{\alpha}_{\proHS}[(I+A), (I+B)])^2 &= \frac{1}{\alpha^2}\left(\trace[(I+A)^{2\alpha}-I] +\trace[(I+B)^{2\alpha}-I] \right)
	\nonumber
	\\
	& - \frac{2}{\alpha^2} \trace([(I+A)^{\alpha}(I+B)^{2\alpha}(I+A)^{\alpha}]^{1/2} - I).
	\label{equation:min-AI-HS-3}
	\end{align}
\end{theorem}

\subsection{The case $\gamma = \nu > 0$}
\label{section:gamma-nu-equal}
The
case $\gamma = \nu=1$ generalizes to the case $\gamma = \nu > 0$ as follows.

\begin{definition}
	[\textbf{Alpha Procrustes distance between positive definite Hilbert-Schmidt operators}]
\label{definition:alpha-procrustes-infinite-2}
Let $\gamma > 0$,
$\alpha \in \R, \alpha \neq 0$ be fixed. 
The Alpha Procrustes distance between two operators $(A+\gamma I), (B+\gamma I) \in \PC_2(\H)$ is defined to be
\begin{align}
\label{equation:alpha-procustes-infinite-2}
&d^{\alpha}_{\proHS}[(A+\gamma I), (B + \gamma I)]
\nonumber 
\\
&= \min_{(I+U) \in \Ubb(\H) \cap \HS_X(\H)}\left\|\frac{(A+\gamma I)^{\alpha} - (B+\gamma I)^{\alpha}(I+U)}{\alpha}\right\|_{\eHS}.
\end{align}
\end{definition}

\begin{remark}
	We will present  the more technically involved case 
	$\gamma \neq \nu$ in the infinite-dimensional setting in a separate work.
\end{remark}


\begin{theorem}
[\textbf{Explicit expression}]
	\label{theorem:Pdistance-AgammaI}
	Let $(A+\gamma I), (B+\gamma I) \in \PC_2(\H)$ be fixed. Let $\alpha \in \R, \alpha \neq 0$ be fixed.
	Then
	\begin{align}
	\label{equation:dproHS-AI}	
	&(d^{\alpha}_{\proHS}[(A+\gamma I), (B+\gamma I)])^2
	\\
	& = \frac{1}{\alpha^2}\trace[(A+\gamma I)^{2\alpha} + (B+\gamma I)^{2\alpha} - 2[(A+\gamma I)^{\alpha}(B+\gamma I)^{2\alpha}(A+\gamma I)^{\alpha}]^{1/2}]
	\nonumber
	\\
	& = \frac{||(A+\gamma I)^{\alpha}||^2_{\eHS}-\gamma^{2\alpha}}{\alpha^2}  + \frac{||(B+\gamma I)^{\alpha}||^2_{\eHS}- \gamma^{2\alpha}}{\alpha^2}
	\\
	& - \frac{2}{\alpha^2}\trace[[(A+\gamma I)^{\alpha}(B+\gamma I)^{2\alpha}(A+\gamma I)^{\alpha}]^{1/2}-\gamma^{\alpha}(A+\gamma I)^{\alpha} - \gamma^{\alpha}(B+\gamma I)^{\alpha}+ \gamma^{2\alpha}I].
	\nonumber
	\end{align}
\end{theorem}

{\bf Special case: Finite-dimensional setting}. For $A,B \in \Sym^{++}(n)$ ($\alpha \neq 0$), or $A,B \in \Sym^{+}(n)$ ($\alpha > 0$), setting $\gamma = 0$ in Eq.(\ref{equation:dproHS-AI}) gives
\begin{align}
(d^{\alpha}_{\proHS}[A, B])^2 = \frac{1}{\alpha^2}\trace[A^{2\alpha} + B^{2\alpha} - 2(A^{\alpha}B^{2\alpha}A^{\alpha})^{1/2}] = (d^{\alpha}_{\proE}[A,B])^2. 
\end{align}

In the case  $A,B$ are positive trace class operators and $\alpha \geq 1/2$, we can set the quantity $\gamma$ in Theorem \ref{theorem:Pdistance-AgammaI} 
to zero, since we have the following.

\begin{lemma}
	\label{lemma:trace-power-1}
	Assume that $A \in \Sym^{+}(\H) \cap \Tr(\H)$. Then $A^{\alpha}\in \Sym^{+}(\H) \cap \HS(\H)$ and $A^{2\alpha} \in \Sym^{+}(\H) \cap \Tr(\H)$ for all $\alpha \geq 1/2$.
\end{lemma}

\begin{lemma}
	\label{lemma:trace-power-2}
	Assume that $A,B \in \Sym^{+}(\H) \cap \Tr(\H)$. Then $(A^{\alpha}B^{2\alpha}A^{\alpha})^{1/2} \in \Sym^{+}(\H) \cap \Tr(\H)$ for all $\alpha \geq 1/2$.
\end{lemma}

\begin{corollary}
	[\textbf{Explicit expression - Positive trace class operators}]
	\label{corollary:explicit-traceclass}
Assume that $A,B \in \Sym^{+}(\H) \cap \Tr(\H)$ and $\alpha \geq 1/2$. Then
\begin{align}
\label{equation:explicit-traceclass}
d^{\alpha}_{\proHS}(A,B) &= 
\lim_{\gamma \approach 0}d^{\alpha}_{\proHS}[(A+\gamma I), (B+\gamma I)]
\nonumber
\\
&= \frac{(\trace[A^{2 \alpha} + B^{2\alpha} - 2(A^{\alpha}B^{2\alpha}A^{\alpha})^{1/2}])^{1/2}}{\alpha}.
\end{align}
\end{corollary}

In particular, for $\alpha = 1/2$, we recover the Bures-Wasserstein distance.
\begin{corollary}
	[\textbf{Special case - Bures-Wasserstein distance}]
Let $A, B \in \Sym^{+}(\H) \cap \Tr(\H)$. Then
\begin{align}
d^{1/2}_{\proHS}(A,B) &= \lim_{\gamma \approach 0}d^{1/2}_{\proHS}[(A+\gamma I), (B+\gamma I)]
\nonumber
\\
&= 2 (\trace[A + B - 2(A^{1/2}BA^{1/2})^{1/2}])^{1/2}.
\end{align}
\end{corollary}

{\bf Limiting case}. Consider now the limiting case $\alpha \approach 0$.
\begin{lemma}
	\label{lemma:alpha-0-logHS-1}
	For $(I+A) \in \PC_2(\H)$,
	\begin{align}
	\lim_{\alpha \approach 0}\frac{||(I+A)^{\alpha}||^2_{\eHS}-1}{\alpha^2} = ||\log(I+A)||^2_{\HS}.
	\end{align}
	For $(A+\gamma I) \in \PC_2(\H)$,
	\begin{align}
	\lim_{\alpha \approach 0}\frac{||(A+\gamma I)^{\alpha}||^2_{\eHS}-\gamma^{2\alpha}}{\alpha^2} = \left\|\log\left(I+\frac{A}{\gamma}\right)\right\|^2_{\HS}.
	\end{align}
\end{lemma}

\begin{lemma}
	\label{lemma:alpha-0-logHS-2}
	For $(I+A), (I+B) \in \PC_2(\H)$,
	\begin{align}
	&\lim_{\alpha \approach 0}\frac{1}{\alpha^2}\trace[[(I+A)^{\alpha}(I+B)^{2\alpha}(I+A)^{\alpha}]^{1/2} - (I+A)^{\alpha} - (I+B)^{\alpha} + I]
	\nonumber
	\\
	& = \trace[\log(I+A)\log(I+B)] = \la \log(I+A), \log(I+B)\ra_{\HS}.
	\end{align}
	For $(A+\gamma I), (B+\gamma I) \in \PC_2(\H)$,
	\begin{align}
	&\lim_{\alpha \approach 0}\frac{1}{\alpha^2}\trace[[(A + \gamma I)^{\alpha}(B + \gamma I)^{2\alpha}(A + \gamma I)^{\alpha}]^{1/2} - \gamma^{\alpha}(A+\gamma I)^{\alpha} - \gamma^{\alpha}(B+\gamma I)^{\alpha} + \gamma^{2\alpha}I]
	\nonumber
	\\
	& = \trace\left[\log\left(I+\frac{A}{\gamma}\right)\log\left(I+\frac{B}{\gamma}\right)\right] = \left\la \log\left(I+\frac{A}{\gamma}\right), \log\left(I+\frac{B}{\gamma}\right)\right\ra_{\HS}.
	\end{align}
\end{lemma}

Combining Lemmas \ref{lemma:alpha-0-logHS-1} and \ref{lemma:alpha-0-logHS-2} with Theorem \ref{theorem:Pdistance-AgammaI}, we
arrive at the following infinite-dimensional generalization of Theorem 
\ref{theorem:limit-alpha-0-finite}.

\begin{theorem}
[\textbf{Limiting case - Log-Hilbert-Schmidt distance}]
\label{theorem:limiting-infinite}
Let $(A+\gamma I), (B+\gamma I) \in \PC_2(\H)$ be fixed. Then
\begin{align}
\lim_{\alpha \approach 0} d^{\alpha}_{\proHS}[(A+\gamma I), (B+\gamma I)]= 
||\log(A+\gamma I) - \log(B+\gamma I)||_{\eHS}.
\end{align}
\end{theorem}

The following, then, is the infinite-dimensional generalization of Theorem \ref{theorem:metric-family-finite}.
\begin{theorem}
	[\textbf{Parametrized family of distances - Positive definite Hilbert-Schmidt operators}]
	\label{theorem:metric-family-infinite-HS}
Let $\gamma > 0$ be fixed.
The function $d^{\alpha}_{\proHS}$, as defined in Eq.(\ref{equation:alpha-procustes-infinite-2}), is a metric 
on the set $\PC_2(\H)(\gamma) = \{A+\gamma I > 0: A \in \HS(\H)\}$
for all $\alpha \in \R$.
\end{theorem}

As a consequence of Corollary 
\ref{corollary:explicit-traceclass}, we obtain directly a family of distances on the set $\Sym^{+}(\H) \cap \Tr(\H)$ of positive trace class operators,
without the factor $\gamma I$, for all $\alpha \geq 1/2$.
\begin{theorem}
[\textbf{Parametrized family of distances - Positive trace class operators}]
\label{theorem:metric-family-infinite-Tr}
The function $d^{\alpha}_{\proHS}$, as given in Eq.(\ref{equation:explicit-traceclass}), is
a metric on the set $\Sym^{+}(\H) \cap \Tr(\H)$, for all $\alpha \geq 1/2$.
\end{theorem}

\subsection{A parametrized family of distances between Gaussian measures on Hilbert spaces}
\label{section:distance-Gaussian}

We now generalize the parametrized family of distances between Gaussian measures on $\R^n$,
as described in Theorem \ref{theorem:metric-Gaussian-Rn}, to the infinite-dimensional setting.
Similar to the Euclidean space setting,
each distance/divergence on the set $\Sym^{+}(\H) \cap \Tr(\H)$ corresponds to a distance/divergence on 
the set of zero-mean Gaussian measures $\Ncal(0,C)$, $C \in \Sym^{+}(\H) \cap \Tr(\H)$, and vice versa. Similarly,
each distance/divergence on the set $\Sym^{++}(\H) \cap \Tr(\H)$ corresponds to a distance/divergence 
on the set of zero-mean {\it non-degenerate} Gaussian measures $\Ncal(0,C)$, $C \in \Sym^{++}(\H) \cap \Tr(\H)$, and vice versa.

Thus, as a consequence of Theorems \ref{theorem:metric-family-infinite-HS} and \ref{theorem:metric-family-infinite-Tr}, we have the following.
\begin{theorem}
	[\textbf{Alpha Procrustes distances between zero-mean Gaussian measures on Hilbert spaces}]
\label{theorem:metric-Gaussian-Hilbert-zeromean}
Let $\H$ be a separable Hilbert space, $C_1, C_2 \in \Sym^{+}(\H)\cap \Tr(\H)$.
Let $\gamma > 0$ be fixed. 
The following is a parametrized family of squared metrics on the set of zero-mean Gaussian measures on $\H$, $\alpha \in \R$,
\begin{align}
&(D^{\alpha}_{\proHS}[\Ncal(0,C_1)||\Ncal(0,C_2)])^2 
\\
& = \frac{1}{4\alpha^2}\trace[(C_1 + \gamma I)^{2\alpha} + (C_2 + \gamma I)^{2\alpha} - 2((C_1+\gamma I)^{\alpha}(C_2+\gamma I)^{2\alpha}(C_1+\gamma I)^{\alpha})^{1/2}].
\nonumber
\end{align}
Furthermor, for $\alpha \geq 1/2$, the following is a parametrized family of squared metrics on the set of zero-mean Gaussian measures on $\H$,
\begin{align}
&(D^{\alpha}_{\proHS}[\Ncal(0,C_1)||\Ncal(0,C_2)])^2 
= \frac{1}{4\alpha^2}\trace[C_1^{2\alpha} + C_2^{2\alpha} - 2(C_1^{\alpha}C_2^{2\alpha}C_1^{\alpha})^{1/2}].
\end{align}
\end{theorem}
To deal with general Gaussian measures on $\H$, we first have the following result.
\begin{lemma} 
	\label{lemma:metric-quadratic-Gaussian}
	Let $\H$ be a separable Hilbert space, $m_1, m_2 \in \H$, $C_1, C_2 \in \Sym^{+}(\H)\cap \Tr(\H)$.
	Let $d_{\mean}$ be a metric on $\H$ and $d_{\cov}$ be  a metric on $\Sym^{+}(\H) \cap \Tr(\H)$.
	Then the following function is a metric on the set of Gaussian measures 
	on $\H$,
	\begin{align}
	D[\Ncal(m_1, C_1)||\Ncal(m_2,C_2)] = \sqrt{d_{\mean}^2(m_1, m_2) + d_{\cov}^2(C_1, C_2)}.
	\end{align}
\end{lemma}

For general Gaussian measures of the form $\Ncal(m, C)$, $m \in \H$,
motivated by Eq.~(\ref{equation:L2-Wasserstein}), Theorems \ref{theorem:metric-family-infinite-HS} and \ref{theorem:metric-family-infinite-Tr}, and Lemma \ref{lemma:metric-quadratic-Gaussian}, we arrive at the following result.

\begin{theorem}
	[\textbf{Alpha Procrustes distances between Gaussian measures on Hilbert spaces}]
	\label{theorem:metric-Gaussian-Hilbert}
	Let $\H$ be a separable Hilbert space, $m_1, m_2 \in \H$, $C_1, C_2 \in \Sym^{+}(\H)\cap \Tr(\H)$.
	Let $\gamma > 0$ be fixed. Let $d_{\mean}$ be a metric on $\H$.
	The following is a parametrized family of squared metrics on the set of Gaussian measures on $\H$,
	$\alpha \in \R$, 
	\begin{align}
	&(D^{\alpha}_{\proHS}[\Ncal(m_1, C_1)||\Ncal(m_2,C_2)])^2 = d_{\mean}^2(m_1, m_2)
	\\
	&+\frac{1}{4\alpha^2}\trace[(C_1 + \gamma I)^{2\alpha} + (C_2 + \gamma I)^{2\alpha} - 2((C_1+\gamma I)^{\alpha}(C_2+\gamma I)^{2\alpha}(C_1+\gamma I)^{\alpha})^{1/2}].
	\nonumber
	\end{align}
	Furthermore, for $\alpha \geq 1/2$, the following is a family of squared metrics  on the set of Gaussian measures on $\H$,
	\begin{align}
	(D^{\alpha}_{\proHS}[\Ncal(m_1, C_1)||\Ncal(m_2,C_2)])^2 &= d_{\mean}^2(m_1, m_2)
	\\
	&+\frac{1}{4\alpha^2}\trace[C_1^{2\alpha} + C_2^{2\alpha} - 2(C_1^{\alpha}C_2^{2\alpha}C_1^{\alpha})^{1/2}].
	\nonumber
	\end{align}
\end{theorem}

{\bf Special cases}. 
For $\alpha =1/2$  and $d_{\mean}(m_1, m_2) = ||m_1 - m_2||$, where $||\;||$ is the Hilbert norm on $\H$,
we recover the {\bf $\L^2$-Wasserstein distance} between two Gaussian measures on $\H$, as given in Eq.(\ref{equation:L2-Wasserstein}).

As $\alpha \approach 0$, we obtain the following squared distance
\begin{align}
&(D^{\alpha}_{\logHS}[\Ncal(m_1,C_1)||\Ncal(m_2,C_2)])^2 = \lim_{\alpha \approach 0}(D^{\alpha}_{\proHS}[\Ncal(m_1, C_1)||\Ncal(m_2,C_2)])^2 
\nonumber
\\
&= d_{\mean}^2(m_1, m_2) + \frac{1}{4}||\log(C_1 + \gamma I) - \log(C_2 + \gamma I)||^2_{\eHS}.
\end{align}

\subsection{The reproducing kernel Hilbert space (RKHS) setting}
\label{section:RKHS}
We now
present
explicit expressions for the Alpha Procrustes distances
between RKHS covariance operators. We first compute distances between 
operators of the form $(AA^{*} + \gamma I), \gamma > 0$, where 
$AA^{*}$ is a trace class operator, which is automatically positive, on a separable Hilbert space.
The main idea here is to express $d^{\alpha}_{\proHS}[(AA^{*} + \gamma I), (BB^{*} + \gamma I )]$
in terms of quantities involving $A^{*}A$ and $B^{*}B$,
which are easier to compute in the RKHS setting.

To state our next result,
let $E:\H_1 \mapto \H_1$ be a self-adjoint, positive, compact operator
on a separable Hilbert space $\H_1$, with nonzero eigenvalues $\{\lambda_k(E)\}_{k=1}^{N_E}$, $1 \leq N_E \leq \infty$,
and corresponding orthonormal eigenvectors $\{\phi_k(E)\}_{k=1}^{N_E}$.
Consider the
following 
operator
\begin{align}
\label{equation:h-alpha}
&h_{\alpha}(E) = \sum_{k=1}^{N_E}\frac{(1+\lambda_k(E))^{\alpha}-1}{\lambda_k(E)}\phi_k(E) \otimes \phi_k(E), \;\;\text{(Lemma 10, \cite{Minh:LogDetIII2018})}.
\end{align}
When $\dim(\H_1) < \infty$, let $E = U_E\Sigma_E U_E^T$ be the reduced
singular value decomposition of $E$, where $U_E$ is a unitary matrix of size $\dim(\H_1) \times N_E$. Then
\begin{align}
\label{equation:h-alpha-finite}
h_{\alpha}(E) = U_E[(\Sigma_E + I_{N_E})^{\alpha} - I_{N_E}]\Sigma_E^{-1}U_E^T.
\end{align}
\begin{proposition}
	\label{proposition:dproHS-H1H2}
	Let $\H_1, \H_2$ be two separable Hilbert spaces.
	Let $A,B: \H_1 \mapto \H_2$ be compact operators 
	such that $A^{*}A, B^{*}B, A^{*}B: \H_1 \mapto \H_1$ are trace class operators. Then	
	\begin{align}
	&\alpha^2(d^{\alpha}_{\proHS}[(I_{\H_2}+AA^{*}), (I_{\H_2}+BB^{*})])^2
	\nonumber
	\\
	& = \trace[(I_{\H_1}+A^{*}A)^{2\alpha} - I_{\H_1}] + \trace[(I_{\H_1}+B^{*}B)^{2\alpha} - I_{\H_1}]
	\nonumber
	\\
	&- 2\trace\left[\left[I_{\H_1^3} 
	+
	\begin{pmatrix}
	C_{11} & C_{12} & C_{13}
	\\
	C_{21} & C_{22} & C_{23}
	\\
	C_{21} & C_{22} & C_{23}
	\end{pmatrix}
	\right]^{1/2} - I_{\H_1^3}\right],
	\end{align}
	where the operators $C_{ij}: \H_1 \mapto \H_1$, $i=1,2$, $j=1,2,3$, are given by
	\begin{align}
	&C_{11} = [(I_{\H_1} + A^{*}A)^{2\alpha} - I_{\H_1}], \; \; C_{12} = A^{*} Bh_{2\alpha}(B^{*}B),
	\\
	&C_{13} =[(I_{\H_1} + A^{*}A)^{2\alpha} - I_{\H_1}]A^{*}Bh_{2\alpha}(B^{*}B),
	\\
	&C_{21} =B^{*}Ah_{2\alpha}(A^{*}A), \; \; C_{22} =[(I_{\H_1} + B^{*}B)^{2\alpha} - I_{\H_1}],
	\\
	&C_{23} = B^{*}Ah_{2\alpha}(A^{*}A)A^{*}Bh_{2\alpha}(B^{*}B).
	\end{align}
\end{proposition}

\begin{corollary}
	\label{corollary:dproHS-H1H2-gamma}
	Let $\H_1, \H_2$ be two separable Hilbert spaces.
	Let $A,B: \H_1 \mapto \H_2$ be compact operators 
	such that $A^{*}A, B^{*}B, A^{*}B: \H_1 \mapto \H_1$ are trace class operators. Then	
	for any $\gamma > 0$,
	\begin{align}
	&\alpha^2(d^{\alpha}_{\proHS}[(AA^{*} + \gamma I_{\H_2}), (BB^{*} + \gamma I_{\H_2})])^2
	\nonumber
	\\
	& = \trace[(A^{*}A + \gamma I_{\H_1})^{2\alpha} - \gamma^{2\alpha}I_{\H_1}] + \trace[(B^{*}B + \gamma I_{\H_1})^{2\alpha} - \gamma^{2\alpha}I_{\H_1}]
	\nonumber
	\\
	&- 2\gamma^{2\alpha}\trace\left[\left[I_{\H_1^3} 
	+
	\begin{pmatrix}
	C_{11} & C_{12} & C_{13}
	\\
	C_{21} & C_{22} & C_{23}
	\\
	C_{21} & C_{22} & C_{23}
	\end{pmatrix}
	\right]^{1/2} - I_{\H_1^3}\right],
	\end{align}
	where the operators $C_{ij}$, $i=1,2$, $j=1,2,3$, are given by
	\begin{align}
	&C_{11} = [(I_{\H_1} + A^{*}A/\gamma)^{2\alpha} - I_{\H_1}], \; \; C_{12} = \frac{1}{\gamma}A^{*} Bh_{2\alpha}(B^{*}B/\gamma),
	\\
	&C_{13} =\frac{1}{\gamma}[(I_{\H_1} + A^{*}A/\gamma)^{2\alpha} - I_{\H_1}]A^{*}Bh_{2\alpha}(B^{*}B/\gamma),
	\\
	&C_{21} =\frac{1}{\gamma}B^{*}Ah_{2\alpha}(A^{*}A/\gamma), \; \; C_{22} =[(I_{\H_1} + B^{*}B/\gamma)^{2\alpha} - I_{\H_1}],
	\\
	&C_{23} = \frac{1}{\gamma^2}B^{*}Ah_{2\alpha}(A^{*}A/\gamma)A^{*}Bh_{2\alpha}(B^{*}B/\gamma).
	\end{align}
\end{corollary}

Since $AA^{*}, BB^{*}$ are positive trace class operators on $\H_2$, by Corollary \ref{corollary:explicit-traceclass}, we have
for $\alpha \geq 1/2$, the simplified expression
\begin{align}
&d^{\alpha}_{\proHS}[(AA^{*}), (BB^{*})] = \lim_{\gamma \approach 0}d^{\alpha}_{\proHS}[(AA^{*}+\gamma I_{\H_2}), (BB^{*}+\gamma I_{\H_2})] 
\nonumber
\\
&=\frac{\trace[(AA^{*})^{2\alpha} + (BB^{*})^{2\alpha} - 2((AA^{*})^{\alpha}(BB^{*})^{2\alpha}(AA^{*})^{\alpha})^{1/2}]^{1/2}}{\alpha}
\nonumber
\\
&= \frac{\trace[(AA^{*})^{2\alpha} + (BB^{*})^{2\alpha} - 2((AA^{*})^{2\alpha}(BB^{*})^{2\alpha})^{1/2}]^{1/2}}{\alpha}.
\end{align}

Consequently, the expression given in Corollary \ref{corollary:dproHS-H1H2-gamma}
can also be considerably simplified as follows.
\begin{corollary}
	\label{corollary:dproHS-H1H2-gamma-0}
	Assume the hypothesis of Corollary \ref{corollary:dproHS-H1H2-gamma}. For $\alpha \geq 1/2$,
	\begin{align}
	&d^{\alpha}_{\proHS}[(AA^{*}), (BB^{*})] =
	\lim_{\gamma \approach 0}d^{\alpha}_{\proHS}[(AA^{*}+\gamma I_{\H_2}), (BB^{*}+\gamma I_{\H_2})] 
	\nonumber
	\\
	&= \frac{1}{\alpha}\trace[(A^{*}A)^{2\alpha} + (B^{*}B)^{2\alpha} - 2(B^{*}A(A^{*}A)^{2\alpha-1}A^{*}B(B^{*}B)^{2\alpha-1})^{1/2}]^{1/2}.
	\end{align}
	In particular, for $\alpha = 1/2$, as can be verified directly,
	\begin{align}
	&d^{1/2}_{\proHS}[(AA^{*}), (BB^{*})] = \lim_{\gamma \approach 0}d^{1/2}_{\proHS}[(AA^{*}+\gamma I_{\H_2}), (BB^{*}+\gamma I_{\H_2})] 
	\nonumber
	\\
	&= 2[\trace(A^{*}A) + \trace(B^{*}B) - 2 \trace(B^{*}AA^{*}B)^{1/2}]^{1/2}
	= 2d_{\BW}[(AA^{*}), (B^{*}B)].
	\end{align}
\end{corollary}

Let {$\X$} now be a separable topological space and 
{$K$} be a continuous positive definite kernel on {$\X \times \X$}. Then the reproducing kernel Hilbert space (RKHS) {$\H_K$}
induced by {$K$} is separable (\cite{Steinwart:SVM2008}, Lemma 4.33).
Let {$\Phi: \X \mapto \H_K$} be the corresponding canonical feature map, so that 
	$K(x,y) = \la \Phi(x), \Phi(y)\ra_{\H_K}$ $\forall (x,y) \in \X \times \X$.
	Assume furthermore that $\X$ is a complete separable metric space. Let $\rho$ be a Borel probability measure on $\X$ such that
		\begin{align}
		\int_{\X}||\Phi(x)||_{\H_K}^2d\rho(x) = \int_{\X}K(x,x)d\rho(x) < \infty.
		\end{align}
	Then the RKHS mean vector $\mu_{\Phi} \in \H_K$ and covariance operator {$C_{\Phi}:\H_K \mapto \H_K$} are both well-defined and are given by
		\begin{align}
		\mu_{\Phi} &= \int_{\X}\Phi(x)d\rho(x) \in \H_K, 
		\\\;\;\;
		C_{\Phi} &= \int_{\X}(\Phi(x)-\mu_{\Phi})\otimes (\Phi(x)-\mu_{\Phi})d\rho(x),
		\end{align}
	where, for $u,v,w \in \H_K, (u\otimes v)w = \la v,w\ra_{\H_K}u$. 
	In particular, the covarince operator {$C_{\Phi}$} is a positive trace class operator on $\H_K$ (see e.g. \cite{Minh:Covariance2017}).
	
Let {$\Xbf =[x_1, \ldots, x_m]$,$m \in \N$,} be a data matrix randomly sampled from {$\X$} according to a Borel probability distribution $\rho$, where {$m \in \N$} is the number of observations.
The feature map {$\Phi$} on {$\Xbf$} 
defines
the bounded linear operator
\begin{align}
	\Phi(\Xbf): \R^m \mapto \H_K, \;\;\; \Phi(\Xbf)\b = \sum_{j=1}^mb_j\Phi(x_j) , \b \in \R^m.
\end{align}
Informally, the operator {$\Phi(\Xbf)$} can also be viewed as the (potentially infinite) mapped feature matrix 
{$\Phi(\Xbf) = [\Phi(x_1), \ldots, \Phi(x_m)]$} of size {$\dim(\H_K) \times m$} in the feature space $\H_K$, with the $j$th column being
{$\Phi(x_j)$}. 
The corresponding empirical mean vector and empirical covariance operator for {$\Phi(\Xbf)$}
  are defined to be
\begin{align}
\label{equation:mean-vector}
\mu_{\Phi(\Xbf}) &= \frac{1}{m}\Phi(\Xbf)\1_m,
\\
C_{\Phi(\Xbf)} &= \frac{1}{m}\Phi(\Xbf)J_m\Phi(\Xbf)^{*}: \H_K \mapto \H_K,
\label{equation:covariance-operator}
\end{align}
where {$J_m = I_m -\frac{1}{m}\1_m\1_m^T$} is the centering matrix,
with
$\1_m = (1, \ldots, 1)^T \in \R^m$.

Let {$\Xbf = [x_i]_{i=1}^m$, $\Ybf = [y_i]_{i=1}^m$}, be two random data matrices sampled from {\small$\X$} according to two Borel probability distributions and $C_{\Phi(\Xbf)}$, $C_{\Phi(\Ybf)}$
be the corresponding covariance operators induced by
$K$. 
Define the $m \times m$ Gram matrices
\begin{align}
K[\Xbf] = \Phi(\Xbf)^{*}\Phi(\Xbf),\;K[\Ybf] = \Phi(\Ybf)^{*}\Phi(\Ybf), K[\Xbf,\Ybf] = \Phi(\Xbf)^{*}\Phi(\Ybf),
\\
(K[\Xbf])_{ij} = K(x_i,xj), \; (K[\Ybf])_{ij} = K(y_i,y_j), \; (K[\Xbf,\Ybf])_{ij} = K(x_i,y_j),
\\
i,j=1, \ldots, m.
\nonumber
\end{align}
Define
$A = \frac{1}{\sqrt{m}}\Phi(\Xbf)J_m:\R^m \mapto \H_K$, $B = \frac{1}{\sqrt{m}}\Phi(\Ybf)J_m:\R^m \mapto \H_K$, so that
\begin{align}
\label{equation:Gram-AB}
	A^{*}A = \frac{1}{m}J_mK[\Xbf]J_m, \; B^{*}B = \frac{1}{m}J_mK[\Ybf]J_m, \;A^{*}B = \frac{1}{m}J_m K[\Xbf,\Ybf]J_m.
\end{align}

\begin{theorem}
	[\textbf{Alpha Procrustes distance between RKHS covariance operators}]
	\label{theorem:proHS-RKHS}	
	Let $A^{*}A$, $A^{*}B$, $B^{*}B$ be as defined in Eq.(\ref{equation:Gram-AB})
	and $h_{\alpha}$ be as defined  in Eq.(\ref{equation:h-alpha-finite}), with $\H_1= \R^m$. Then
\begin{align}
\label{equation:proHS-RKHS}
&\alpha^2 (d^{\alpha}_{\proHS}[(C_{\Phi(\Xbf)} + \gamma I_{\H_K}), (C_{\Phi(\Ybf)} + \gamma I_{\H_K})])^2 
\\
& = \trace[(A^{*}A + \gamma I_{m})^{2\alpha} - \gamma^{2\alpha}I_{m}] + \trace[(B^{*}B + \gamma I_{m})^{2\alpha} - \gamma^{2\alpha}I_{m}]
\nonumber
\\
&- 2\gamma^{2\alpha}\trace\left[\left[I_{3m} 
+
\begin{pmatrix}
C_{11} & C_{12} & C_{13}
\\
C_{21} & C_{22} & C_{23}
\\
C_{21} & C_{22} & C_{23}
\end{pmatrix}
\right]^{1/2} - I_{3m}\right],
\end{align}
where the $m \times m$ matrices $C_{ij}$, $i=1,2$, $j=1,2,3$, are given by
{\small
\begin{align}
&C_{11} = [(I_{m} + A^{*}A/\gamma)^{2\alpha} - I_{m}], \; \; C_{12} = \frac{1}{\gamma}A^{*} Bh_{2\alpha}(B^{*}B/\gamma),
\\
&C_{13} =\frac{1}{\gamma}[(I_{m} + A^{*}A/\gamma)^{2\alpha} - I_{m}]A^{*}Bh_{2\alpha}(B^{*}B/\gamma),
\\
&C_{21} =\frac{1}{\gamma}B^{*}Ah_{2\alpha}(A^{*}A/\gamma), \; \; C_{22} =[(I_{m} + B^{*}B/\gamma)^{2\alpha} - I_{m}],
\\
&C_{23} = \frac{1}{\gamma^2}B^{*}Ah_{2\alpha}(A^{*}A/\gamma)A^{*}Bh_{2\alpha}(B^{*}B/\gamma).
\end{align}
}
\end{theorem}
\subsection{Parametrized family of distances between Gaussian measures in RKHS}
\label{section:distance-Gaussian-RKHS}

We now combine the results in Section \ref{section:distance-Gaussian} on the distances between Gaussian measures
on a separable Hilbert space with the results in Section \ref{section:RKHS} to obtain
a parametrized family of distances between Gaussian measures in an RKHS induced by a positive definite kernel $K$,
with closed form formulas via the corresponding kernel Gram matrices

Throughout the following, let $K$, $\H_K$, $\Xbf$, $\Ybf$, $\mu_{\Phi(\Xbf)},\mu_{\Phi(\Ybf)}$, and $C_{\Phi(\Xbf)}, C_{\Phi(\Ybf)}$ be as defined
in Section \ref{section:RKHS}. 

\begin{theorem}
	[\textbf{Alpha Procrustes distances between Gaussian measures on RKHS}]
	\label{theorem:distance-Gaussian-RKHS-alpha}
	Let $\alpha \in \R$ and $d_{\mean} = ||\;||_{\H_K}$. Let $\gamma > 0$ be fixed.
	The following is a parametrized family of squared distances, with parameter $\alpha$, between
	the Gaussian measures $\Ncal(\mu_{\Phi(\Xbf)}, C_{\Phi(\Xbf)})$ and $\Ncal(\mu_{\Phi(\Ybf)}, C_{\Phi(\Ybf)})$,
	\begin{align}
	&\left(D^{\alpha}_{\proHS}[\Ncal(\mu_{\Phi(\Xbf)}, C_{\Phi(\Xbf)}) || \Ncal(\mu_{\Phi(\Ybf)}, C_{\Phi(\Ybf)})]\right)^2
	\nonumber
	\\
	& = ||\mu_{\Phi(\Xbf)} - \mu_{\Phi(\Ybf)}||^2_{\H_K}+ \frac{1}{4} (d^{\alpha}_{\proHS}[(C_{\Phi(\Xbf)} + \gamma I_{\H_K}), (C_{\Phi(\Ybf)} + \gamma I_{\H_K})])^2
	\\
	& = \frac{1}{m^2}\1_m^T(K[\Xbf] + K[\Ybf] - 2K[\Xbf, \Ybf])\1_m
	\\
	&+\frac{1}{4} (d^{\alpha}_{\proHS}[(C_{\Phi(\Xbf)} + \gamma I_{\H_K}), (C_{\Phi(\Ybf)} + \gamma I_{\H_K})])^2,
	\nonumber
	\end{align}
	where $d^{\alpha}_{\proHS}[(C_{\Phi(\Xbf)} + \gamma I_{\H_K}), (C_{\Phi(\Ybf)} + \gamma I_{\H_K})]$ 
	is as given in Eq.(\ref{equation:proHS-RKHS}).
\end{theorem}

\begin{theorem}
	[\textbf{Alpha Procrustes distances between Gaussian measures on RKHS - $\alpha \geq 1/2$}]
	\label{theorem:distance-Gaussian-RKHS-alpha-more-than1/2}
Let $\alpha \geq 1/2$ and $d_{\mean} = ||\;||_{\H_K}$.
The following is a parametrized family of distances, with parameter $\alpha$, between
the Gaussian measures $\Ncal(\mu_{\Phi(\Xbf)}, C_{\Phi(\Xbf)})$ and $\Ncal(\mu_{\Phi(\Ybf)}, C_{\Phi(\Ybf)})$.
\begin{align}
&\left(D^{\alpha}_{\proHS}[\Ncal(\mu_{\Phi(\Xbf)}, C_{\Phi(\Xbf)})|| \Ncal(\mu_{\Phi(\Ybf)}, C_{\Phi(\Ybf)})]\right)^2
\nonumber
\\
&=||\mu_{\Phi(\Xbf)} - \mu_{\Phi(\Ybf)}||^2_{\H_K} + \frac{1}{4\alpha^2}\trace[C_{\Phi(\Xbf)}^{2\alpha} + C_{\Phi(\Ybf)}^{2\alpha} - 2 (C_{\Phi(\Xbf)}^{\alpha}C_{\Phi(\Ybf)}^{2\alpha}C_{\Phi(\Xbf)}^{\alpha})^{1/2}]
\\
&=
\frac{1}{m^2}\1_m^T[K[\Xbf] + K[\Ybf] - 2K[\Xbf,\Ybf]]\1_m
\\
&\quad+ \frac{1}{4\alpha^2m^{2\alpha}}\trace(J_mK[\Xbf]J_m)^{2\alpha} +  \frac{1}{4\alpha^2m^{2\alpha}}\trace(J_mK[\Ybf]J_m)^{2\alpha}
\nonumber
\\
&\quad- \frac{2}{4\alpha^2m^{2\alpha}}\trace[J_mK[\Ybf,\Xbf]J_m(J_mK[\Xbf]J_m)^{2\alpha-1}J_mK[\Xbf,\Ybf]J_m(J_mK[\Ybf]J_m)^{2\alpha-1}]^{1/2}.
\nonumber
\end{align}
\end{theorem}
In particular, for $\alpha = 1/2$, we obtain the $\L^2$-Wasserstein distance.
\begin{corollary}
	[\textbf{$\L^2$-Wasserstein distance between Gaussian measures in RKHS}]
	\label{corollary:distance-Gaussian-L2Wasserstein}
\begin{align}
\label{equation:distance-Gaussian-L2Wasserstein}
&(D^{1/2}_{\proHS}[\Ncal(\mu_{\Phi(\Xbf)}, C_{\Phi(\Xbf)}) || \Ncal(\mu_{\Phi(\Ybf)}, C_{\Phi(\Ybf)})])^2
\nonumber
\\
&=||\mu_{\Phi(\Xbf)} - \mu_{\Phi(\Ybf)}||^2_{\H_K} + \trace[C_{\Phi(\Xbf)} + C_{\Phi(\Ybf)} - 2 (C_{\Phi(\Xbf)}^{1/2}C_{\Phi(\Ybf)}C_{\Phi(\Xbf)}^{1/2})^{1/2}]
\\
&=
\frac{1}{m^2}\1^T_m[K[\Xbf] + K[\Ybf] - 2K[\Xbf,\Ybf]]\1_m
\\
&\quad+ \frac{1}{m}\trace(J_mK[\Xbf]J_m) +  \frac{1}{m}\trace(J_mK[\Ybf]J_m)
- \frac{2}{m}\trace[J_mK[\Ybf,\Xbf]J_mK[\Xbf,\Ybf]J_m]^{1/2}.
\nonumber
\nonumber
\end{align}
\end{corollary}

\begin{remark} To keep our expressions simple, we have assumed that the number of data points in 
	the matrices $\Xbf$ and $\Ybf$ are the same. If $\Xbf = [x_i]_{i=1}^m$, and $\Ybf = [y_i]_{i=1}^n$, 
	then instead of Eq.(\ref{equation:distance-Gaussian-L2Wasserstein}), we have the following more general expression 
	\begin{align}
	\label{equation:distance-Gaussian-L2Wasserstein-mn}
	&(D^{1/2}_{\proHS}[\Ncal(\mu_{\Phi(\Xbf)}, C_{\Phi(\Xbf)}) || \Ncal(\mu_{\Phi(\Ybf)}, C_{\Phi(\Ybf)})])^2
	\nonumber
	\\
	&=||\mu_{\Phi(\Xbf)} - \mu_{\Phi(\Ybf)}||^2_{\H_K} + \trace[C_{\Phi(\Xbf)} + C_{\Phi(\Ybf)} - 2 (C_{\Phi(\Xbf)}^{1/2}C_{\Phi(\Ybf)}C_{\Phi(\Xbf)}^{1/2})^{1/2}]
	\\
	&=
	\frac{1}{m^2}\1^T_mK[\Xbf]\1_m + \frac{1}{n^2}\1_n^TK[\Ybf]\1_n - \frac{2}{mn}\1^T_mK[\Xbf,\Ybf]\1_n
	\\
	&\quad+ \frac{1}{m}\trace(J_mK[\Xbf]J_m) +  \frac{1}{n}\trace(J_nK[\Ybf]J_n)
	- \frac{2}{\sqrt{mn}}\trace[J_nK[\Ybf,\Xbf]J_mK[\Xbf,\Ybf]J_n]^{1/2}.
	\nonumber
	\nonumber
	\end{align}
	
\end{remark}

\begin{remark}
	Numerical experiments utilizing the above RKHS distances will be presented in a separate work.
\end{remark}

%
%


\bibliographystyle{splncs04}
\bibliography{cite_RKHS}

\begin{thebibliography}{10}
\providecommand{\url}[1]{\texttt{#1}}
\providecommand{\urlprefix}{URL }
\providecommand{\doi}[1]{https://doi.org/#1}

\bibitem{Araki1990inequality}
Araki, H.: On an inequality of {Lieb} and {Thirring}. Letters in Mathematical
  Physics  \textbf{19}(2),  167--170 (1990)

\bibitem{LogEuclidean:SIAM2007}
Arsigny, V., Fillard, P., Pennec, X., Ayache, N.: Geometric means in a novel
  vector space structure on symmetric positive-definite matrices. SIAM J. on
  Matrix An. and App.  \textbf{29}(1),  328--347 (2007)

\bibitem{Bhatia:2018Bures}
Bhatia, R., Jain, T., Lim, Y.: On the {Bures--Wasserstein} distance between
  positive definite matrices. Expositiones Mathematicae  (2018)

\bibitem{Downson:1982}
Dowson, D., Landau, B.: The {Fr\'echet} distance between multivariate normal
  distributions. Journal of Multivariate Analysis  \textbf{12}(3),  450 -- 455
  (1982)

\bibitem{Dryden:2009}
Dryden, I., Koloydenko, A., Zhou, D.: Non-{E}uclidean statistics for covariance
  matrices, with applications to diffusion tensor imaging. Annals of Applied
  Statistics  \textbf{3},  1102--1123 (2009)

\bibitem{Gelbrich:1990Wasserstein}
Gelbrich, M.: On a formula for the {L2} {Wasserstein} metric between measures
  on {Euclidean} and {Hilbert} spaces. Mathematische Nachrichten
  \textbf{147}(1),  185--203 (1990)

\bibitem{Givens:1984}
Givens, C.R., Shortt, R.M.: A class of {Wasserstein} metrics for probability
  distributions. Michigan Math. J.  \textbf{31}(2),  231--240 (1984)

\bibitem{Larotonda:2007}
Larotonda, G.: Nonpositive curvature: A geometrical approach to
  {H}ilbert-{S}chmidt operators. Differential Geometry and its Applications
  \textbf{25},  679--700 (2007)

\bibitem{Lieb1976inequalities}
Lieb, E., Thirring, W.: Inequalities for the moments of the eigenvalues of the
  {S}chr{\"o}dinger {Hamiltonian} and their relation to {Sobolev} inequalities.
  In: Lieb, E., S.B., Wrightman, A. (eds.) Studies in Mathematical Physics.
  Princeton University Press (1976)

\bibitem{Malago:WassersteinGaussian2018}
Malag{\`o}, L., Montrucchio, L., Pistone, G.: Wasserstein {Riemannian} geometry
  of {Gaussian} densities. Information Geometry  \textbf{1}(2),  137--179 (Dec
  2018)

\bibitem{Masarotto:2018Procrustes}
Masarotto, V., Panaretos, V., Zemel, Y.: Procrustes metrics on covariance
  operators and optimal transportation of {Gaussian} processes. Sankhya A pp.
  1--42 (2018)

\bibitem{MinhSB:NIPS2014}
Minh, H.Q., Biagio, M.S., Murino, V.: Log-{H}ilbert-{S}chmidt metric between
  positive definite operators on {H}ilbert spaces. In: Advances in Neural
  Information Processing Systems 27 (NIPS 2014), pp. 388--396 (2014)

\bibitem{Minh:LogDet2016}
Minh, H.: Infinite-dimensional {Log-Determinant} divergences between positive
  definite trace class operators. Linear Algebra and Its Applications
  \textbf{528},  331--383 (2017)

\bibitem{Minh:LogDetIII2018}
Minh, H.: Infinite-dimensional {Log-Determinant} divergences {III}:
  {Log-Euclidean} and {Log-Hilbert--Schmidt} divergences. In: Information
  Geometry and its Applications IV. pp. 209--243. Springer (2018)

\bibitem{Minh:GSI2019}
Minh, H.: A unified formulation for the {Bures-Wasserstein} and
  {Log-Euclidean/Log-Hilbert-Schmidt} distances between positive definite
  operators. In: International Conference on Geometric Science of Information.
  Springer (2019)

\bibitem{Minh:Covariance2017}
Minh, H., Murino, V.: Covariances in computer vision and machine learning.
  Synthesis Lectures on Computer Vision  \textbf{7}(4),  1--170 (2017)

\bibitem{Olkin:1982}
Olkin, I., Pukelsheim, F.: The distance between two random vectors with given
  dispersion matrices. Linear Algebra and its Applications  \textbf{48},  257
  -- 263 (1982)

\bibitem{Steinwart:SVM2008}
Steinwart, I., Christmann, A.: Support vector machines. Springer Science \&
  Business Media (2008)

\bibitem{Takatsu2011wasserstein}
Takatsu, A.: Wasserstein geometry of {Gaussian} measures. Osaka Journal of
  Mathematics  \textbf{48}(4),  1005--1026 (2011)

\bibitem{Villani:2008Optimal}
Villani, C.: Optimal transport: old and new, vol.~338. Springer Science \&
  Business Media (2008)

\bibitem{Wang1995trace}
Wang, B.Y., Zhang, F.: Trace and eigenvalue inequalities for ordinary and
  {Hadamard} products of positive semidefinite {Hermitian} matrices. SIAM
  journal on matrix analysis and applications  \textbf{16}(4),  1173--1183
  (1995)

\end{thebibliography}

\appendix

\section{Proofs for the finite-dimensional case}
\label{section:proofs-finite}

\begin{proof}
	[\textbf{of Theorem \ref{theorem:limit-alpha-0-finite}}]
	We rewrite the expression for $d^{\alpha}_{\proE}(A,B)$ as 
	\begin{align*}
	(d^{\alpha}_{\proE}(A,B))^2 
	&= \frac{||A^{\alpha}-I||^2_F}{\alpha^2} + \frac{||B^{\alpha} - I||^2_F}{\alpha^2}
	\\
	& - \frac{2}{\alpha^2}\trace[(A^{\alpha}B^{2\alpha}A^{\alpha})^{1/2} - A^{\alpha} - B^{\alpha} + I].
	\end{align*}
	We have for the first two terms
	\begin{align*}
	\lim_{\alpha \approach 0}\frac{||A^{\alpha} - I||_F^2}{\alpha^2} = ||\log(A)||^2_F, \;\;\; 
	\lim_{\alpha \approach 0}\frac{||B^{\alpha} - I||_F^2}{\alpha^2} = ||\log(B)||^2_F.
	\end{align*}
	For the last term,
	\begin{align*}
	A^{\alpha} &= \exp(\alpha\log(A)) = I + \alpha \log(A) + \frac{\alpha^2}{2!}(\log(A))^2 + \cdots,
	\\
	B^{2\alpha } &= \exp(2\alpha \log(B)) = I + 2\alpha \log(B) + \frac{(2\alpha)^2}{2!}(\log(B))^2 + \cdots,
	\\
	A^{\alpha}B^{2\alpha}  &= I + \alpha [\log(A) + 2\log(B)] \\
	&+ \alpha^2\left(\frac{1}{2!}(\log(A))^2 + 2\log(A)\log(B) + \frac{2^2}{2!}(\log(B))^2\right) + \cdots
	\\
	A^{\alpha}B^{2\alpha}A^{\alpha} &= I + 2\alpha(\log(A) + \log(B)) + 2\alpha^2(\log(A) + \log(B))^2 + \cdots.
	\end{align*}	
	For $||A|| < 1$, the following series is absolutely convergent
	\begin{align*}
	\sqrt{I + A} = I + \frac{1}{2}A - \frac{1}{8}A^2 + \cdots 
	\end{align*}
	Thus it follows that for $\alpha$ sufficiently small
	\begin{align*}
	&(A^{\alpha}B^{2\alpha}A^{\alpha})^{1/2} = I + \alpha(\log(A) + \log(B)) + \frac{\alpha^2}{2}(\log(A) + \log(B))^2 + \cdots,
	\end{align*}
	so that
	\begin{align*}
	&(A^{\alpha}B^{2\alpha}A^{\alpha})^{1/2} - A^{\alpha} - B^{\alpha} + I = \frac{\alpha^2}{2}[\log(A)\log(B) + \log(B)\log(A)] + \cdots.
	\end{align*}
	It follows that
	\begin{align*}
	&\lim_{\alpha \approach 0}\frac{1}{\alpha^2}\trace[(A^{\alpha}B^{2\alpha}A^{\alpha})^{1/2} - A^{\alpha} - B^{\alpha} + I]
	= \trace[\log(A)\log(B) + \log(B)\log(A)] 
	\\
	&= 2\trace[\log(A)\log(B)] = 2\la \log(A), \log(B)\ra_F.
	\end{align*}
	Combining with the limits of the first two terms, we obtain
	\begin{align*}
	\lim_{\alpha \approach 0}(d^{\alpha}_{\proE}(A,B))^2 &= ||\log(A)||^2_F +||\log(B)||^2_F - 2\la \log(A), \log(B)\ra_F
	\\
	 &= ||\log(A) - \log(B)||^2_F.
	\end{align*}
	This completes the proof. \qed
\end{proof}

\begin{proof}
	[\textbf{of Theorem \ref{theorem:metric-family-finite}}]
	For $\alpha =0$, the Log-Euclidean distance is a metric on $\Sym^{++}(n)$.
	For $\alpha \neq 0$, the metric properties of $d^{\alpha}_{\proE}$ follow from the metric properties of the Bures-Wasserstein distance.
	\qed
\end{proof}

\begin{proof}
	[\textbf{of Theorem \ref{theorem:power-euclidean-procrustes-compare}}]
	From the inequality in Eq.(\ref{equation:Araki-Lieb-Thirring}), we have for any parir
	$A,B \in \Sym^{+}(n)$, $\alpha \neq 0$, (with $A,B \in \Sym^{++}(n)$ if $\alpha < 0$),
	\begin{align}
	\trace(A^{\alpha}B^{2\alpha}A^{\alpha})^{1/2} \geq \trace(A^{\alpha}B^{\alpha}),
	\end{align}
	with equality if and only if $AB = BA$. It follows that
	\begin{align*}
	&d^2_{E, \alpha}(A,B) = \left\|\frac{A^{\alpha} - B^{\alpha}}{\alpha}\right\|^2_F 
	= \frac{1}{\alpha^2}\trace[A^{2\alpha} + B^{2\alpha} - 2A^{\alpha}B^{\alpha}]
	\\
	& \geq \frac{1}{\alpha^2}\trace[A^{2\alpha} + B^{2\alpha} - 2(A^{\alpha}B^{2\alpha}A^{\alpha})^{1/2}] = d^{\alpha}_{\proE}(A,B).
	\end{align*}
	This completes the proof.\qed
\end{proof}

\subsection{Proofs for the Riemannian metric}

\begin{proof}
	[\textbf{of Proposition \ref{proposition:differential}}]
	For $A_0 \in \GL(n)$, the differential map $Dg(A_0): \Mrm(n) \mapto \Sym(n)$ is given by
	the action 
	\begin{align}
	Dg(A_0)(X) = \alpha^2(XA_0^{*} + A_0X^{*}), \;\;\; X \in \Mrm(n).
	\end{align}
	By the chain rule, we have
	\begin{align*}
	D\pi_{\alpha}(A_0)(X) = \frac{\alpha}{2}D\exp\left(\frac{1}{2\alpha}\log(\alpha^2 A_0A_0^{*})\right) \compose D\log(\alpha^2 A_0A_0^{*}) \compose Dg(A_0)(X).
	\end{align*}
	For any $A_0 \in \Sym^{++}(n)$, $D\log(A_0):\Sym(n) \mapto \Sym(n)$ and $D\exp(\log(A_0)): \Sym(n) \mapto \Sym(n)$ are invertible operators, thus
	they both have zero null spaces. It follows that
	\begin{align*}
	&\ker(D\pi_{\alpha}(A_0)) = \ker (Dg(A_0)) = \{X \in \Mrm(n): XA_0^{*} + A_0X^{*} = 0\}
	\\
	& = \{X \in \Mrm(n): XA_0^{*} \;\text{is skew-symmetric}\} = (\Sym(n))^{\perp}(A_0^{*})^{-1}.
	\end{align*}
	Its orthogonal complement is thus the space
	\begin{align*}
	&\ker(D\pi_{\alpha}(A_0))^{\perp} = \Sym(n)A_0.
	\end{align*}
	It follows from this expression that $D\pi_{\alpha}(A_0):\Mrm(n) \mapto \Sym(n)$ must be surjective, hence
	$\pi_{\alpha}$ is a smooth submersion. \qed
\end{proof}

\begin{proof}
	[\textbf{of Proposition \ref{proposition:general-lyapunov}}]
	For any $S \in \Sym(n)$, there is a unique matrix $H \in \Sym(n)$ satisfying the Lyapunov equation
	\begin{align*}
	HP_0^{2\alpha} + P_0^{2\alpha}H = S.
	\end{align*}
	Since $D\exp(\log(P_0)): \Sym(n) \mapto \Sym(n)$ and $D\log(P_0^{2\alpha}):\Sym(n) \mapto \Sym(n)$ are both invertible operators, it follows that
	for any $Y \in \Sym(n)$, there is a unique matrix $H \in \Sym(n)$ such that
	\begin{align*}
	D\exp(\log(P_0)) \compose D\log(P_0^{2\alpha})(HP_0^{2\alpha} + P_0^{2\alpha}H) = Y.
	\end{align*}   
	This completes the proof.
	\qed
\end{proof}

\begin{proof}
	[\textbf{of Theorem \ref{theorem:general-metric}}]
	We recall that the mapping $\pi_{\alpha}:\GL(n) \mapto \Sym^{++}(n)$ is a Riemannian submersion if it is a smooth submersion and
	that the differential map $D\pi_{\alpha}(A_0): \H_{A_0} \mapto T_{\pi_{\alpha}(A_0)}\Sym^{++}(n)$ is an isometry $\forall A_0 \in \GL(n)$.
	
	Let $P_0 = (\alpha^2 AA_0^{*})^{1/(2\alpha)} \equivalent \alpha^2 A_0A_0^{*} = P_0^{2\alpha}$, where $A_0 \in GL(n)$. 
	Let us define an inner product on the tangent space $T_{P_0}(\Sym^{++}(n)) = \Sym(n)$ so that the mapping $\pi_{\alpha}: \GL(n) \mapto \Sym^{++}(n)$ is a Riemannian submersion.
	
	On the horizontal space $\H_{A_0}$, the inner product between two elements $HA_0$ and $KA_0$, $H, K \in \Sym(n)$, is given by
	\begin{align*}
	&\la HA_0, KA_0\ra_F= \trace(A_0^{*}HKA_0) = \trace(KA_0A_0^{*}H) = \frac{1}{\alpha^2}\trace(KP_0^{2\alpha}H).
	\end{align*}
	On the other hand,
	\begin{align*}
	&D\pi_{\alpha}(A_0)(HA_0) =\frac{\alpha}{2}D\exp\left(\frac{1}{2\alpha}\log(\alpha^2 A_0A_0^{*})\right) \compose D\log(\alpha^2 A_0A_0^{*}) (HA_0A_0^{*} + A_0A_0^{*}H)
	\\
	& = \frac{1}{2\alpha}D\exp(\log(P_0))\compose D\log(P_0^{2\alpha})(HP_0^{2\alpha} + P_0^{2\alpha}H)
	\\
	&D\pi_{\alpha}(A_0)(KA_0) =   \frac{1}{2\alpha}D\exp(\log(P_0))\compose D\log(P_0^{2\alpha})(KP_0^{2\alpha} + P_0^{2\alpha}K).
	\end{align*}
	Thus in order to have an isometry, we need 
	\begin{align*}
	& \frac{1}{4}\la D\exp(\log(P_0))\compose D\log(P_0^{2\alpha})(HP_0^{2\alpha} + P_0^{2\alpha}H), D\exp(\log(P_0))\compose D\log(P_0^{2\alpha})(KP_0^{2\alpha} + P_0^{2\alpha}K)\ra_{P_0}
	\\
	&= \trace(KP_0^{2\alpha}H).
	\end{align*}
	By Proposition \ref{proposition:general-lyapunov}, for any $Y \in \Sym(n)$, there is a unique $H \in \Sym(n)$ such that
	\begin{align*}
	D\exp(\log(P_0))\compose D\log(P_0^{2\alpha})(HP_0^{2\alpha} + P_0^{2\alpha}H) = Y.
	\end{align*}
	Recall that we denote such an $H$ by $H = \Lcal_{P_0, \alpha}(Y)$. Then the isometry requirement is equivalent to
	\begin{align*}
	\la Y, Z\ra_{P_0} = 4\trace(\Lcal_{P_0, \alpha}(Z)P_0^{2\alpha}\Lcal_{P_0, \alpha}(Y)) = 4\trace(\Lcal_{P_0, \alpha}(Y)P_0^{2\alpha}\Lcal_{P_0, \alpha}(Z)). 
	\end{align*}
	This gives the desired Riemannian metric on $\Sym^{++}(n)$.
	\qed
\end{proof}

\begin{proof}
	[\textbf{of Theorem \ref{theorem:geodesic}}]
	Let $U$ be the unitary polar factor for $B^{\alpha}A^{\alpha}$, so that
	\begin{align*}
	B^{\alpha}A^{\alpha} = U(A^{\alpha}B^{2\alpha}A^{\alpha})^{1/2}
	\end{align*}
	Then $U$ can be expressed as
	\begin{align*}
	&U = B^{\alpha}A^{\alpha}(A^{\alpha}B^{2\alpha}A^{\alpha})^{-1/2} 
	= B^{\alpha}A^{\alpha}(A^{\alpha}B^{2\alpha}A^{\alpha})^{-1/2}A^{-\alpha}A^{\alpha}
	\\
	& = B^{\alpha}(A^{2\alpha}B^{2\alpha})^{-1/2}A^{\alpha} = B^{-\alpha}B^{2\alpha}(B^{-2\alpha}A^{-2\alpha})^{1/2}A^{\alpha}
	= B^{-\alpha}(A^{-2\alpha} \# B^{2\alpha})A^{\alpha}.
	\end{align*}
	Define the following curve
	\begin{align*}
	Z(t) &= \frac{1}{|\alpha|}[(1-t)A^{\alpha} + t B^{\alpha}U] = \frac{1}{|\alpha|}[(1-t)A^{\alpha} + t(A^{-2\alpha} \# B^{2\alpha})A^{\alpha}]
	\\
	& = \frac{1}{|\alpha|}[(1-t)I + t(A^{-2\alpha} \# B^{2\alpha})]A^{\alpha},
	\end{align*}
	with $Z(0) = \frac{1}{|\alpha|}A^{\alpha}$ and $Z(1) = \frac{1}{|\alpha|}(A^{-2\alpha} \# B^{2\alpha})A^{\alpha}$. Note that $Z(t) \in GL(n)$ for all $0 \leq t \leq 1$, since it is a product of two SPD matrices. Also, since $Z(t)$ is a straight line segment 
	joining $A^{\alpha}$ and $(A^{-2\alpha} \# B^{2\alpha})A^{\alpha}$, it is a geodesic curve in $\GL(n)$. Its derivative with respect to $t$ is given by
	\begin{align*}
	Z{'}(t) = \frac{1}{|\alpha|}[-I + (A^{-2\alpha} \# B^{2\alpha})]A^{\alpha} \in \H_{A^{\alpha}}.
	\end{align*}
	Thus $Z{'}(t)$ is horizontal for all $t$ and hence by Theorem \ref{theorem:geodesic-submersion}, $\gamma = \pi_{\alpha}\compose Z$ is a geodesic in 
	the Riemannian manifold $\Sym^{++}(n)$, with
	\begin{align*}
	\gamma(0) = (\alpha^2 Z(0)Z(0))^{1/(2\alpha)} = A, 
	\gamma(1) = (\alpha^2 Z(1)Z(1))^{1/(2\alpha)} = B.
	\end{align*}
	It follows that $\gamma$ is a geodesic joining $A$ and $B$ in $\Sym^{++}(n)$, with
	\begin{align*}
	&\gamma(t) = (\alpha^2Z(t)Z(t)^{*})^{1/(2\alpha)}
	\\
	&= [(1-t)^2A^{2\alpha} + t^2B^{2\alpha} + t(1-t)[(A^{2\alpha}B^{2\alpha})^{1/2} + (B^{2\alpha}A^{2\alpha})^{1/2}]]^{1/(2\alpha)}.
	\end{align*}
	Furthermore, by Theorem \ref{theorem:geodesic-submersion}, the length of $\gamma = \pi_{\alpha} \compose Z$ is the same as 
	that of $Z$, which, being in a straight line in $(\GL(n), \la \;,\;\ra_F)$, is given by
	\begin{align*}
	L(\gamma) = L(Z) = \frac{1}{|\alpha|}||A^{\alpha} -B^{\alpha}U||_F 
	= \frac{1}{|\alpha|}[\trace(A^{2\alpha}+B^{2\alpha} - 2(A^{\alpha}B^{2\alpha}A^{\alpha}))]^{1/2}
	\end{align*}
	where the last equality follows from Theorem 1 in \cite{Bhatia:2018Bures}.
	
	By employing a similar argument with horizontal curve lifting as in \cite{Bhatia:2018Bures}, this is also the minimum possible length of any curve joining $A,B \in \Sym^{++}(n)$. Thus it must be the Riemmanian distance $d(A,B)$. \qed
\end{proof}

\section{Proofs for the infinite-dimensional case}
\label{section:proofs-infinite}

\begin{proof}
	[\textbf{of Lemma \ref{lemma:polar-decomp-HSX}}]
	Consider the polar decomposition when $\gamma = 1$,
	\begin{align*}
	& I+A = S|I+A| = S[(I+A^{*})(I+A)]^{1/2}.
	\end{align*}
	Since, necessarily, $|I+A| = [(I+A^{*})(I+A)]^{1/2} \in \PC_2(\H)$, we must have $S = (I+A)|I+A|^{-1} = I+R \in \HS_X(\H)$,
	where $R \in \HS_X(\H)$. Therefore,
	\begin{align*}
	&S^{*}S  = |I+A|^{-1}(I+A^{*})(I+A)|I+A|^{-1} = I,
	\\
	& SS^{*} = (I+A)|I+A|^{-2}(I+A^{*}) = I.
	\end{align*}
	The general case reduces to the case $\gamma =1 $ since
	\begin{align*}
	& (A+\gamma I) = S|A+\gamma I| \imply S = (A+\gamma I)|A+\gamma I|^{-1} = (A/\gamma + I)|A/\gamma + I|^{-1}.
	\end{align*}
	This completes the proof. \qed
\end{proof}

\begin{lemma}
	\label{lemma:I+U}
Assume that $I+U \in \Ubb(\H) \cap \HS_X(\H)$. Then
$U^{*} + U = - U^{*}U = - UU^{*} \in \Tr(\H)$.
\end{lemma}
\begin{proof}
	Since $U \in \HS(\H)$, we have $U^{*}U, UU^{*} \in \Tr(\H)$. Since $(I+U) \in \Ubb(\H)$,
\begin{align*}
&I = (I+U)^{*}(I+U) = I + U^{*} + U + U^{*}U \imply U^{*} + U = - U^{*}U \in \Tr(\H).
\end{align*}
Similarly, $U^{*} + U = - UU^{*} \in \Tr(\H)$. \qed
\end{proof}

\begin{lemma}
	\label{lemma:I+A-trace}
	Let $(I+A) \in \PC_2(\H)$ be fixed. 
	Let $\alpha \in \R, \alpha > 0$ be fixed.
	For $(I+U) \in \Ubb(\H) \cap \HS_X(\H)$,
	\begin{align}
	(I+A)^{2\alpha} - (I+A)^{\alpha}(I+U^{*}) - (I+U)(I+A)^{\alpha} + I \in \Tr(\H),
	\end{align} 
	and
	\begin{align}
	&||(I+A)^{\alpha}(I+U)||^2_{\eHS}
	\nonumber
	\\
	& = 1 + \trace[(I+A)^{2\alpha} - (I+A)^{\alpha}(I+U^{*}) - (I+U)(I+A)^{\alpha} + I]
	\\
	& = 1 + \trace[(I+A)^{2\alpha} - (I+A)^{\alpha}(I+U^{*}) - (I+A)^{\alpha}(I+U) + I].
	\end{align}
	For $U = 0$,
	\begin{align}
	& ||(I+A)^{\alpha}||^2_{\eHS} = 1 + \trace[(I+A)^{2\alpha} - 2(I+A)^{\alpha} + I].
	\end{align}
\end{lemma}
\begin{proof}
	Write
	$(I+A)^{\alpha} = I +B$, where  $B \in \HS(\H)$.
	By Lemma \ref{lemma:I+U},
	\begin{align*}
	&(I+A)^{2\alpha} - (I+A)^{\alpha}(I+U^{*}) - (I+U)(I+A)^{\alpha} + I
	\\
	& = (I+B)^2 - (I+B)(I+U^{*}) - (I+U)(I+B)
	\\
	& = (I+2B + B^2) - (I+B +U^{*} + BU^{*}) - (I+U+B +UB) + I
	\\
	& = B^2 +U^{*}U  - BU^{*} - UB  \in \Tr(\H),\;\;\;\text{since $B \in \HS(\H), U \in \HS(\H)$}.
	\end{align*}
	It follows that
	\begin{align*}
	&||(I+A)^{\alpha}(I+U)||^2_{\eHS} = ||[(I+A)^{\alpha}(I+U)- I] + I||^2_{\eHS} 
	\\
	& = ||[(I+A)^{\alpha}(I+U)- I]||^2_{\HS} + 1 
	\\
	&= 1 + \trace[(I+U^{*})(I+A)^{2\alpha}(I+U)- (I+U^{*})(I+A)^{\alpha} - (I+A)^{\alpha}(I+U) + I]
	\\
	& = 1 + \trace[(I+U^{*})[(I+A)^{2\alpha} - (I+A)^{\alpha}(I+U)^{*} - (I+U)(I+A)^{\alpha} + I](I+U)]
	\\
	& = 1 + \trace[(I+A)^{2\alpha} - (I+A)^{\alpha}(I+U^{*}) - (I+U)(I+A)^{\alpha} + I]
	\\
	& = 1 + \trace[(I+A)^{2\alpha} - (I+A)^{\alpha}(I+U^{*}) - (I+A)^{\alpha}(I+U) + I],
	\end{align*}
	where the last equality follows from $\trace(AB) = \trace(BA)$ for $A,B$ compact.\qed
\end{proof}

\begin{proof}
	[\textbf{of Proposition \ref{proposition:limit-UV-same}}]
	On the one hand, by Lemma \ref{lemma:I+A-trace},
	\begin{align*}
	&||(I+A)^{\alpha}(I+U) - (I+B)^{\alpha}(I+V)||^2_{\eHS}
	\\
	& = ||(I+A)^{\alpha}(I+U)||^2_{\eHS} + (I+B)^{\alpha}(I+V)||^2_{\eHS}
	\\
	& - 2 \la (I+A)^{\alpha}(I+U), (I+B)^{\alpha}(I+V)\ra_{\eHS}
	\\
	& =  \trace[(I+A)^{2\alpha} - (I+A)^{\alpha}(I+U^{*}) - (I+A)^{\alpha}(I+U) + I]
	\\
	&  + \trace[(I+B)^{2\alpha} - (I+B)^{\alpha}(I+V^{*}) - (I+B)^{\alpha}(I+V) + I]
	\\
	&- 2\la [(I+A)^{\alpha}(I+U) - I], [(I+B)^{\alpha}(I+V)-I]\ra_{\HS}
	\\
	& = \trace[(I+A)^{2\alpha} - (I+A)^{\alpha}(I+U^{*}) - (I+A)^{\alpha}(I+U) + I]
	\\
	&  + \trace[(I+B)^{2\alpha} - (I+B)^{\alpha}(I+V^{*}) - (I+V)^{\alpha}(I+V) + I]
	\\
	&- \trace[(I+U^{*})(I+A)^{\alpha} - I][(I+B)^{\alpha}(I+V)-I]]
	\\
	& - \trace[(I+V^{*})(I+B)^{\alpha} - I][(I+A)^{\alpha}(I+U)- I]
	\\
	& = \trace[(I+A)^{2\alpha} - (I+A)^{\alpha}(I+U^{*}) - (I+A)^{\alpha}(I+U) + I]
	\\
	&  + \trace[(I+B)^{2\alpha} - (I+B)^{\alpha}(I+V^{*}) - (I+B)^{\alpha}(I+V) + I]
	\\
	&- \trace[(I+U^{*})(I+A)^{\alpha}(I+B)^{\alpha}(I+V) - (I+U^{*})(I+A)^{\alpha} - (I+B)^{\alpha}(I+V) + I]
	\\
	& - \trace[(I+V^{*})(I+B)^{\alpha}(I+A)^{\alpha}(I+U) - (I+V^{*})(I+B)^{\alpha} - (I+A)^{\alpha}(I+U) + I]
	\\
	& = \trace[(I+A)^{2\alpha} - (I+A)^{\alpha}(I+U^{*}) - (I+A)^{\alpha}(I+U) + I]
	\\
	&  + \trace[(I+B)^{2\alpha} - (I+B)^{\alpha}(I+V^{*}) - (I+B)^{\alpha}(I+V) + I]
	\\
	&- \trace[(I+V)(I+U^{*})(I+A)^{\alpha}(I+B)^{\alpha} - (I+A)^{\alpha}(I+U^{*}) - (I+B)^{\alpha}(I+V) + I]
	\\
	& - \trace[(I+U)(I+V^{*})(I+B)^{\alpha}(I+A)^{\alpha} - (I+B)^{\alpha}(I+V^{*}) - (I+A)^{\alpha}(I+U) + I]
	\\
	& = \trace[(I+A)^{2\alpha} + (I+B)^{2\alpha} - (I+V)(I+U^{*})(I+A)^{\alpha}(I+B)^{\alpha} - (I+U)(I+V^{*})(I+B)^{\alpha}(I+A)^{\alpha}].
	\end{align*}
	On the other hand, since $(I+A)^{\alpha} - (I+B)^{\alpha}(I+V) \in \HS(\H)$, we have
	\begin{align*}
	&||(I+A)^{\alpha} - (I+B)^{\alpha}(I+V)||^2_{\eHS} = ||(I+A)^{\alpha} - (I+B)^{\alpha}(I+V)||^2_{\HS}
	\\
	&= \trace[((I+A)^{\alpha} - (I+V)^{*}(I+B)^{\alpha})((I+A)^{\alpha} - (I+B)^{\alpha}(I+V))]
	\\
	& = \trace[(I+A)^{2\alpha} - (I+A)^{\alpha}(I+B)^{\alpha}(I+V) - (I+V)^{*}(I+B)^{\alpha}(I+A)^{\alpha} + (I+V)^{*}(I+B)^{2\alpha}(I+V)]
	\\
	& = \trace[(I+A)^{2\alpha} +  (I+B)^{2\alpha}- (I+V)(I+A)^{\alpha}(I+B)^{\alpha} - (I+V^{*})(I+B)^{\alpha}(I+A)^{\alpha}].
	\end{align*}
	Thus minimizing the first expression over $(I+U),(I+V) \in \Ubb(\H) \cap \HS_X(\H)$ is equivalent to minimizing the second
	expression over  $(I+U) \in \Ubb(\H) \cap \HS_X(\H)$. 
	 \qed
\end{proof}

\begin{proposition}
	\label{proposition:HS-to-Tr}
	Let $(I+A), (I+B) \in \HS_X(\H)$ be invertible. Then
	\begin{align}
	&(I+A)^2 + (I+B)^2 - 2[(I+A)(I+B)^2(I+A)]^{1/2} \in \Tr(\H),
	\\
	&[(I+A)(I+B)^2(I+A)]^{1/2} - (I+A) - (I+B) + I \in \Tr(\H).
	\end{align}
\end{proposition}
\begin{proof}
	[\textbf{of Proposition \ref{proposition:HS-to-Tr}}]
	Write $I+C = (I+A)(I+B)^2(I+A) = [(I+B)(I+A)]^{*}[(I+B)(I+A)]$. By Lemma \ref{lemma:polar-decomp-HSX},
	the polar decomposition of $(I+B)(I+A)$ has the form
	\begin{align*}
	(I+B)(I+A) = (I+S)(I+C)^{1/2}, \;\;\;\text{for some $I+S \in \Ubb(\H) \cap \HS_X(\H)$},
	\end{align*}
	so that
	\begin{align*}
	(I+C)^{1/2} = (I+S^{*})(I+B)(I+A) = (I+A)(I+B)(I+S)
	\end{align*}
	It follows that
	\begin{align*}
	& (I+A)^2 + (I+B)^2 - 2[(I+A)(I+B)^2(I+A)]^{1/2} 
	\\
	&= (I+A)^2 + (I+B)^2 - (I+S^{*})(I+B)(I+A) - (I+A)(I+B)(I+S)
	\\
	& = (I+2A + A^2)+(I + 2B + B^2) - (I+S^{*})(I+A+B + BA) - (I+A+B + AB)(I+S)
	\\
	& = A^2 + B^2 - (S^{*} + S^{*}A + S^{*}B + BA + S^{*}BA) - (S +AS +BS + AB + ABS)
	\\
	& = A^2 + B^2 + S^{*}S - (S^{*}A + S^{*}B + BA + S^{*}BA) - (AS +BS + AB + ABS) \in \Tr(\H)
	\end{align*}
	since $A, B, S \in \HS(\H)$ and $S+S^{*} = - S^{*}S \in \Tr(\H)$ by Lemma \ref{lemma:I+U}.
	Similarly,
	\begin{align*}
	&[(I+A)(I+B)^2(I+A)]^{1/2} - (I+A) - (I+B)  + I 
	\\
	& = \frac{1}{2}(I+S^{*})(I+B)(I+A) +\frac{1}{2}(I+A)(I+B)(I+S)- (I+A) - (I+B)  + I 
	\\
	& = \frac{1}{2}(S^{*}S + S^{*}A + S^{*}B + BA + S^{*}BA + AS +BS + AB + ABS) \in \Tr(\H).
	\end{align*}
	This completes the proof.\qed
\end{proof}

\begin{proof}
	[\textbf{of Proposition \ref{proposition:HS-to-Tr-alpha}}]
	This follows from Proposition \ref{proposition:HS-to-Tr}
	by replacing $(I+A)$ with $(I+A)^{\alpha}$. The condition
	$(I+A) \in \PC_2(\H)$ is necessary so that
	$(I+A)^{\alpha}$ is well-defined for all $\alpha \in \R$. 
	\qed
\end{proof}

\begin{proof}	
	[\textbf{ of Theorem \ref{theorem:Pdistance-AI-HS}}]
	\begin{align*}
	&||(I+A)^{\alpha} - (I+B)^{\alpha}(I+U)||^2_{\eHS} = ||(I+A)^{\alpha} - (I+B)^{\alpha}(I+U)||^2_{\HS} 
	\\
	&= \trace[((I+A)^{\alpha} - (I+U)^{*}(I+B)^{\alpha})((I+A)^{\alpha} - (I+B)^{\alpha}(I+U))]
	\\
	& = \trace[(I+A)^{2\alpha} - (I+A)^{\alpha}(I+B)^{\alpha}(I+U) - (I+U)^{*}(I+B)^{\alpha}(I+A)^{\alpha} + (I+U)^{*}(I+B)^{2\alpha}(I+U)]
	\\
	& = \trace[(I+A)^{2\alpha} + (I+B)^{2\alpha} - (I+A)^{\alpha}(I+B)^{\alpha}(I+U) - (I+U^{*})(I+B)^{\alpha}(I+A)^{\alpha}].
	\end{align*}
	Consider the operator
	\begin{align*}
	I+C = [(I+B)^{\alpha}(I+A)^{\alpha}]^{*}(I+B)^{\alpha}(I+A)^{\alpha} = (I+A)^{\alpha}(I+B)^{2\alpha}(I+A)^{\alpha} \in \PC_2(\H).
	\end{align*}
	By Lemma \ref{lemma:polar-decomp-HSX}, the polar decomposition of $(I+B)^{\alpha}(I+A)^{\alpha}$ has the form
	\begin{align*}
	(I+B)^{\alpha}(I+A)^{\alpha} = (I+S)(I+C)^{1/2}, \;\;\text{for some $I+S \in \Ubb(\H) \cap \HS_X(\H)$}.
	\end{align*}
	It follows that
	\begin{align*}
	&(I+U^{*})(I+B)^{\alpha}(I+A)^{\alpha} = (I+U^{*})(I+S)(I+C)^{1/2},
	\\
	&(I+A)^{\alpha}(I+B)^{\alpha}(I+U) = (I+C)^{1/2}(I+S^{*})(I+U).
	\end{align*}
	Since $(I+U^{*})(I+S) \in \Ubb(\H) \cap \HS_X(\H)$, it follows that
	\begin{align*}
	&\min_{(I+U) \in \Ubb(\H) \cap \HS_X(\H)} ||(I+A)^{\alpha} - (I+B)^{\alpha}(I+U)||^2_{\eHS}
	\\
	& = \min_{(I+U) \in \Ubb(\H) \cap \HS_X(\H)} \trace[(I+A)^{2\alpha} + (I+B)^{2\alpha} - (I+U^{*})(I+C)^{1/2}- (I+C)^{1/2}(I+U)]
	\\
	& = \min_{(I+U) \in \Ubb(\H) \cap \HS_X(\H)}\trace[(I+A)^{2\alpha} + (I+B)^{2\alpha} - (I+U^{*})(I+C)^{1/2}- (I+U)(I+C)^{1/2}].
	\end{align*}
	By Proposition \ref{proposition:HS-to-Tr-alpha},
	\begin{align*}
	&(I+A)^{2\alpha} + (I+B)^{2\alpha} - 2(I+C)^{1/2} 
	\\
	&= (I+A)^{2\alpha} + (I+B)^{2\alpha} - 2[(I+A)^{\alpha}(I+B)^{2\alpha}(I+A)^{\alpha}]^{1/2}
	\in \Tr(\H).
	\end{align*}
	Since $I+U \in \Ubb(\H) \cap \HS_X(\H)$, by Lemma \ref{lemma:I+U}, $U^{*} + U = - U^{*}U \in \Tr(H)$. Thus
	\begin{align*}
	&\min_{(I+U) \in \Ubb(\H) \cap \HS_X(\H)} ||(I+A)^{\alpha} - (I+B)^{\alpha}(I+U)||^2_{\eHS}
	\\
	&= \trace[(I+A)^{2\alpha} + (I+B)^{2\alpha} - 2(I+C)^{1/2}]
	\\
	& + \min_{(I+U) \in \Ubb(\H) \cap \HS_X(\H)}\trace[U^{*}U(I+C)^{1/2}]
	\\
	& = \trace[(I+A)^{2\alpha} + (I+B)^{2\alpha} - 2(I+C)^{1/2}],
	\end{align*}
	since obviously, with $\trace[U^{*}U(I+C)^{1/2}] = \trace[U^{*}(I+C)^{1/2}U] = ||(I+C)^{1/4}U||^2_{\HS}$,
	\begin{align*}
	&\min_{(I+U) \in \Ubb(\H) \cap \HS_X(\H)}\trace[U^{*}U(I+C)^{1/2}] = \min_{(I+U) \in \Ubb(\H) \cap \HS_X(\H)}||(I+C)^{1/4}U||^2_{\HS} = 0,
	\end{align*}
	with the minimum occurring precisely at $U = 0$. This gives Eq.(\ref{equation:min-AI-HS-1}).
	The expression in Eq.(\ref{equation:min-AI-HS-2}) then follows from Lemma \ref{lemma:I+A-trace}.
	\qed
	\end{proof}

The following is immediate from the definition of $d^{\alpha}_{\proHS}$.

\begin{lemma}
	\label{lemma:Pdistance-decompose}
\begin{align}
d^{\alpha}_{\proHS}[(A+\gamma I), (B+\gamma I)] = \gamma^{\alpha}d^{\alpha}_{\proHS}
[(A/\gamma + I), (B/\gamma + I)].
\end{align}
\end{lemma}

\begin{proof}
	[\textbf{of Theorem \ref{theorem:Pdistance-AgammaI}}]
	By Lemma \ref{lemma:Pdistance-decompose} and Theorem \ref{theorem:Pdistance-AI-HS},
	\begin{align*}
	&(d^{\alpha}_{\proHS}[(A+\gamma I), (B+ \gamma I)]) = \gamma^{2\alpha}(d^{\alpha}_{\proHS}
	[(A/\gamma + I), (B/\gamma + I)])^2
	\\
	& = \frac{\gamma^{2\alpha}}{\alpha^2}\left(\trace[(I+A/\gamma)^{2\alpha} + (I+B/\gamma)^{2\alpha} -2 [(I+A/\gamma)^{\alpha}(I+B/\gamma)^{2\alpha}(I+A/\gamma)^{\alpha}]^{1/2}]\right)
	\\
	& = \frac{1}{\alpha^2}\left(\trace[(A+ \gamma I)^{2\alpha} + (B + \gamma I)^{2\alpha} -2 [(A + \gamma I)^{\alpha}(B + \gamma I)^{2\alpha}(A  + \gamma I)^{\alpha}]^{1/2}]\right).
	\end{align*}
	Similarly,
	\begin{align*}
	&(d^{\alpha}_{\proHS}[(A+\gamma I), (B+ \gamma I)])
	\\
	& = \frac{\gamma^{2\alpha}}{\alpha^2}\left(||(I+A/\gamma)^{\alpha}||^2_{\eHS} + ||(I+B/\gamma)^{\alpha}||^2_{\eHS} - 2\right)
	\nonumber
	\\
	& - \frac{2\gamma^{2\alpha}}{\alpha^2}\left(\trace[[(I+A/\gamma)^{\alpha}(I+B/\gamma)^{2\alpha}(I+A/\gamma)^{\alpha}]^{1/2} - (I+A/\gamma)^{\alpha} - (I+B/\gamma)^{\alpha} + I]\right)
	\\
	& = \frac{1}{\alpha^2}\left(||(A + \gamma I)^{\alpha}||^2_{\eHS} + ||(B + \gamma I)^{\alpha}||^2_{\eHS} - 2\gamma^{2\alpha}\right)
	\\
	& - \frac{2}{\alpha^2}\left(\trace[[(A + \gamma I)^{\alpha}(B + \gamma I)^{2\alpha}(A + \gamma I)^{\alpha}]^{1/2} - \gamma^{\alpha}(A + \gamma I)^{\alpha} - \gamma^{\alpha}(B + \gamma I)^{\alpha} + \gamma^{2\alpha}I]\right).
	\end{align*}
	This completes the proof.
	\qed
\end{proof}

\begin{proof}
	[\textbf{of Lemma \ref{lemma:alpha-0-logHS-1}}]
	Consider the series expansion
\begin{align*}
(I+A)^{\alpha}  = \exp[\alpha\log(I+A)] = I + \alpha \log(I+A) + \frac{\alpha^2}{2!}[\log(I+A)]^2 + \cdots
\end{align*}
From this expansion it follows that
\begin{align*}
\frac{(I+A)^{\alpha} - I}{\alpha} - \log(I+A) = \frac{\alpha}{2!}[\log(I+A)]^2 + \frac{\alpha^2}{3!}[\log(I+A)]^3 + \cdots
\end{align*}
Therefore, since the set $\HS(\H)$ is an algebra,
\begin{align*}
&\left\|\frac{(I+A)^{\alpha} - I}{\alpha} - \log(I+A)\right\|_{\HS}
\\
& \leq |\alpha|||\log(I+A)||^2_{\HS}\left(\frac{1}{2!}
+\frac{|\alpha|}{3!}||\log(I+A)||_{\HS} + \frac{|\alpha|^2}{4!}||\log(I+A)||^2_{\HS} + \cdots\right)
\\
& \leq |\alpha|||\log(I+A)||^2_{\HS}\exp(|\alpha|||\log(I+A)||_{\HS}) \approach 0 \;\text{as $\alpha \approach \infty$}.
\end{align*}
By definition of the extended Hilbert-Schmidt norm,
\begin{align*}
||(I+A)^{\alpha}||^2_{\eHS} = ||[(I+A)^{\alpha} - I] + I||^2_{\eHS} = ||(I+A)^{\alpha} - I||^2_{\HS} + 1.
\end{align*}
It thus follows that
\begin{align*}
\lim_{\alpha \approach 0}\frac{||(I+A)^{\alpha}||^2_{\eHS} -1}{\alpha^2}
= \lim_{\alpha \approach 0}\frac{||(I+A)^{\alpha} - I||^2_{\HS}}{\alpha^2} = ||\log(I+A)||^2_{\HS}.
\end{align*}	
For $(A+\gamma I) \in \PC_2(\H)$, we then have
\begin{align*}
&\lim_{\alpha \approach 0}\frac{||(A+\gamma I)^{\alpha}||^2_{\eHS}-\gamma^{2\alpha}}{\alpha^2} = \lim_{\alpha \approach 0}\gamma^{2\alpha}\frac{||(A/\gamma+ I)^{\alpha}||^2_{\eHS}-1}{\alpha^2}
\\
& = ||\log(I+A/\gamma)||^2_{\HS}.
\end{align*}
This completes the proof.
\qed
\end{proof}

\begin{proof}
	[\textbf{of Lemma \ref{lemma:alpha-0-logHS-2}}]
	By Proposition \ref{proposition:HS-to-Tr-alpha}, for $(I+A), (I+B) \in \PC_2(\H)$,
	\begin{align*}
	[(I+A)^{\alpha}(I+B)^{2\alpha}(I+A)^{\alpha}]^{1/2} - (I+A)^{\alpha} - (I+B)^{\alpha} + I \in \Tr(\H),
	\end{align*}
so that the trace operation is well-defined. Consider the series expansions
\begin{align*}
&(I+A)^{\alpha} = \exp[\alpha\log(I+A)] = I + \alpha \log(I+A) + \frac{\alpha^2}{2!}[\log(I+A)]^2 + \cdots,
\\
&(I+B)^{\alpha} = \exp[\alpha\log(I+B)] = I + \alpha \log(I+B) + \frac{\alpha^2}{2!}[\log(I+B)]^2 + \cdots,
\\
&(I+A)^{\alpha}(I+B)^{2\alpha}(I+A)^{\alpha} = I + 2\alpha[\log(I+A) + \log(I+B)] 
\\
&+ 2\alpha^2[\log(I+A) + \log(I+B)]^2 + \cdots.
\end{align*}
For $\alpha$ sufficiently close to zero, 
\begin{align*}
[(I+A)^{\alpha}(I+B)^{2\alpha}(I+A)^{\alpha}]^{1/2} &= I + \alpha[\log(I+A) + \log(I+B)] 
\\
&+ \frac{\alpha^2}{2}[\log(I+A) + \log(I+B)]^2 + \cdots.
\end{align*}
It thus follows that
\begin{align*}
&\lim_{\alpha \approach 0}\frac{1}{\alpha^2}\trace[[(I+A)^{\alpha}(I+B)^{2\alpha}(I+A)^{\alpha}]^{1/2} - (I+A)^{\alpha} - (I+B)^{\alpha} + I]
\nonumber
\\
& = \frac{1}{2}\trace[\log(I+A)\log(I+B) + \log(I+B)\log(I+A)] = \trace[\log(I+A)\log(I+B)].
\end{align*}
For the general case $(A+\gamma I), (B+\gamma I) \in \PC_2(\H)$, we have
\begin{align*}
&\lim_{\alpha \approach 0}\frac{1}{\alpha^2}\trace[[(A + \gamma I)^{\alpha}(B + \gamma I)^{2\alpha}(A + \gamma I)^{\alpha}]^{1/2} - \gamma^{\alpha}(A+\gamma I)^{\alpha} - \gamma^{\alpha}(B+\gamma I)^{\alpha} + \gamma^{2\alpha}I]
\\
& = \lim_{\alpha \approach 0}\frac{\gamma^{2\alpha}}{\alpha^2}\trace[[(A/\gamma +  I)^{\alpha}(B/\gamma  +  I)^{2\alpha}(A/\gamma +  I)^{\alpha}]^{1/2} - (A/\gamma+ I)^{\alpha} - (B/\gamma+ I)^{\alpha} + I]
\\
& = \trace[\log(I+A/\gamma)\log(I+B/\gamma)],
\end{align*}
where we have invoked the result for the case $\gamma = 1$ above. 
\qed
\end{proof}

\begin{proof}
	[\textbf{of Theorem \ref{theorem:limiting-infinite}}]
	From Theorem \ref{theorem:Pdistance-AgammaI} and Lemmas \ref{lemma:alpha-0-logHS-1} and \ref{lemma:alpha-0-logHS-2}, we have
	\begin{align*}
	&\lim_{\alpha \approach 0}(d^{\alpha}_{\logHS}[(A+\gamma I), (B+ \gamma I)])^2
	\\
	&	 = \lim_{\alpha \approach 0}\frac{||(A+\gamma I)^{\alpha}||^2_{\eHS}-\gamma^{2\alpha}}{\alpha^2}
	 + \lim_{\alpha \approach 0}\frac{||(B+\gamma I)^{\alpha}||^2_{\eHS}- \gamma^{2\alpha}}{\alpha^2}
	\\
	& - \lim_{\alpha \approach 0}\frac{2}{\alpha^2}\trace[[(A+\gamma I)^{\alpha}(B+\gamma I)^{2\alpha}(A+\gamma I)^{\alpha}]^{1/2}-\gamma^{\alpha}(A+\gamma I)^{\alpha} - \gamma^{\alpha}(B+\gamma I)^{\alpha}+ \gamma^{2\alpha}I]
	\\
	& = ||\log(I+A/\gamma)||^2_{\HS} + ||\log(I+B/\gamma)||^2_{\HS} - 2\la \log(I+A/\gamma), \log(I+B/\gamma)\ra_{\HS}
	\\
	& = ||\log(I+A/\gamma) - \log(I+B/\gamma)||^2_{\HS} = ||\log(A+\gamma I) - \log(B+\gamma I)||^2_{\HS}.
	\end{align*}	
	This completes the proof. \qed
\end{proof}

\begin{proof}
[\textbf{of Theorem \ref{theorem:metric-family-infinite-HS}}]
For $\alpha = 0$, $d^{0}_{\logHS} = ||\;||_{\eHS}$ is a metric on $\PC_2(\H)(\gamma)$ since the Log-Hilbert-Schmidt distance is a metric
\cite{MinhSB:NIPS2014}. Consider the case $\alpha \neq 0$. It suffices to prove for the case $\gamma = 1$. 
Since we already have positivity and symmetry, it remains to prove the triangle inequality.

Let $(I+A), (I+B) \in \PC_2(\H)$, $(I+U), (I+V) \in \Ubb(\H) \cap \eHS(\H)$, then,
as in the proof of Proposition \ref{proposition:limit-UV-same},
\begin{align*}
&||(I+A)^{\alpha}(I+U) - (I+B)^{\alpha}(I+V)||^2_{\HS}
\\
& = \trace[(I+A)^{2\alpha} + (I+B)^{2\alpha} - (I+V)(I+U^{*})(I+A)^{\alpha}(I+B)^{\alpha} 
\\
&\quad - (I+U)(I+V^{*})(I+B)^{\alpha}(I+A)^{\alpha}]
\\
& = ||(I+A)^{\alpha} - (I+B)^{\alpha}(I+V)(I+U^{*})||^2_{\HS}.
\end{align*}
Thus for any other operator $(I+C) \in \PC_2(\H)$,
\begin{align*}
&||(I+A)^{\alpha} - (I+B)^{\alpha}(I+U)||_{\HS} 
\\
&\leq ||(I+A)^{\alpha} - (I+C)^{\alpha}(I+V)||_{\HS} + ||(I+C)^{\alpha}(I+V) - (I+B)^{\alpha}(I+U)||_{\HS}
\\
&  = ||(I+A)^{\alpha} - (I+C)^{\alpha}(I+V)||_{\HS} + ||(I+C)^{\alpha} - (I+B)^{\alpha}(I+U)(I+V^{*})||_{\HS}
\end{align*}
Taking the infimum over all $(I+U), (I+V) \in \Ubb(\H) \cap \eHS(\H)$, we obtain
\begin{align*}
d^{\alpha}_{\proHS}[(I+A), (I+B)] \leq d^{\alpha}_{\proHS}[(I+A), (I+C)] + d^{\alpha}_{\proHS}[(I+C), (I+B)].
\end{align*}
This completes the proof.
\qed
\end{proof}

\subsection{Proofs for the RKHS setting}

\begin{proof}
	[\textbf{of Proposition \ref{proposition:dproHS-H1H2}}]
	By Theorem \ref{theorem:Pdistance-AI-HS}, we have
	\begin{align*}
	&\alpha^2(d^{\alpha}_{\proHS}[(I+AA^{*}), (I+BB^{*})])^2 =
	\trace[(I+AA^{*})^{2\alpha} - I] + \trace[(I+BB^{*})^{2\alpha} - I]
	\\
	& - 2\trace[[(I+AA^{*})^{\alpha}(I+BB^{*})^{2\alpha}(I+AA^{*})^{\alpha}]^{1/2}-I] 
	\\
	& = \trace[(I+AA^{*})^{2\alpha} - I] + \trace[(I+BB^{*})^{2\alpha} - I]
	- 2\trace[[(I+AA^{*})^{2\alpha}(I+BB^{*})^{2\alpha}]^{1/2}-I]. 
	\end{align*}
	Since the non-zero eigenvalues of $AA^{*}$ and $A^{*}A$ are the same, the non-zero eigenvalues
	of $(I+AA^{*})^{2\alpha} - I$ and $(I+A^{*}A)^{2\alpha}- I$ are also the same and
	\begin{align*}
	&\trace[(I_{\H_2}+AA^{*})^{2\alpha} - I_{\H_2}] = \trace[(I_{\H_1}+A^{*}A)^{2\alpha} - I_{\H_1}],
	\\
	&\trace[(I_{\H_2}+BB^{*})^{2\alpha} - I_{\H_2}] = \trace[(I_{\H_1}+B^{*}B)^{2\alpha} - I_{\H_1}].
	\end{align*}
	We make use of the following properties 
	\begin{align}
	Ah_{\alpha}(A) = h_{\alpha}(A)A = (I+A)^{\alpha} - I, \; \text{(Lemma 10, \cite{Minh:LogDetIII2018})},
	\\
	(AA^{*} + I_{\H_2})^{\alpha} - I_{\H_2} = Ah_{\alpha}(A^{*}A)A^{*}, \; \text{(Corollary 2, \cite{Minh:LogDetIII2018})}.
	\end{align}
	It follows that
	\begin{align*}
	&(I_{\H_2} + AA^{*})^{2\alpha} = I_{\H_2} + Ah_{2\alpha}(A^{*}A)A^{*},
	\\
	&(I_{\H_2} + BB^{*})^{2\alpha} = I_{\H_2} + Bh_{2\alpha}(B^{*}B)B^{*},
	\end{align*}
	where $h_{2\alpha}(A^{*}A): \H_1 \mapto \H_1$ and $h_{2\alpha}(B^{*}B): \H_1 \mapto \H_1$.
	It follows that
	\begin{align*}
	&[(I_{\H_2} + AA^{*})^{2\alpha}(I_{\H_2} + BB^{*})^{2\alpha}]^{1/2} - I_{\H_2}
	\\
	& = [(I_{\H_2} + Ah_{2\alpha}(A^{*}A)A^{*})(I_{\H_2} + Bh_{2\alpha}(B^{*}B)B^{*})]^{1/2} - I_{\H_2}
	\\
	& = [I_{\H_2} + Ah_{2\alpha}(A^{*}A)A^{*} + Bh_{2\alpha}(B^{*}B)B^{*} + Ah_{2\alpha}(A^{*}A)A^{*}Bh_{2\alpha}(B^{*}B)B^{*}]^{1/2} - I_{\H_2}.
	\end{align*}
	Consider the operators
	\begin{align*}
	& 
	\begin{pmatrix}
	Ah_{2\alpha}(A^{*}A) & Bh_{2\alpha}(B^{*}B) &  Ah_{2\alpha}(A^{*}A)A^{*}Bh_{2\alpha}(B^{*}B)
	\end{pmatrix}
	: \H_1^3 \mapto \H_2,
	\\
	& \begin{pmatrix}
	A^{*}
	\\
	B^{*}
	\\
	B^{*}
	\end{pmatrix}: \H_2 \mapto \H_1^3.
	\end{align*}
	Then the nonzero eigenvalues of
	\begin{align*}
	&Ah_{2\alpha}(A^{*}A)A^{*} + Bh_{2\alpha}(B^{*}B)B^{*} + Ah_{2\alpha}(A^{*}A)A^{*}Bh_{2\alpha}(B^{*}B)B^{*}
	\\
	&= 
	\begin{pmatrix}
	Ah_{2\alpha}(A^{*}A) & Bh_{2\alpha}(B^{*}B) &  Ah_{2\alpha}(A^{*}A)A^{*}Bh_{2\alpha}(B^{*}B)
	\end{pmatrix}
	\begin{pmatrix}
	A^{*}
	\\
	B^{*}
	\\
	B^{*}
	\end{pmatrix}:
	\H_2 \mapto \H_2
	\end{align*}
	are the same as those of the operator
	\begin{align*}
	& \begin{pmatrix}
	A^{*}
	\\
	B^{*}
	\\
	B^{*}
	\end{pmatrix}
	\begin{pmatrix}
	Ah_{2\alpha}(A^{*}A) & Bh_{2\alpha}(B^{*}B) &  Ah_{2\alpha}(A^{*}A)A^{*}Bh_{2\alpha}(B^{*}B)
	\end{pmatrix}: \H_1^3 \mapto \H_1^3
	\\
	& 
	=
	\begin{pmatrix}
	A^{*}Ah_{2\alpha}(A^{*}A) & A^{*} Bh_{2\alpha}(B^{*}B) & A^{*}Ah_{2\alpha}(A^{*}A)A^{*}Bh_{2\alpha}(B^{*}B)
	\\
	B^{*}Ah_{2\alpha}(A^{*}A) & B^{*} Bh_{2\alpha}(B^{*}B) & B^{*}Ah_{2\alpha}(A^{*}A)A^{*}Bh_{2\alpha}(B^{*}B)
	\\
	B^{*}Ah_{2\alpha}(A^{*}A) & B^{*}Bh_{2\alpha}(B^{*}B) & B^{*}Ah_{2\alpha}(A^{*}A)A^{*}Bh_{2\alpha}(B^{*}B)
	\end{pmatrix}
	\\
	& =
	\begin{pmatrix}
	[(I_{\H_1} + A^{*}A)^{2\alpha} - I_{\H_1}] & A^{*} Bh_{2\alpha}(B^{*}B) & [(I_{\H_1} + A^{*}A)^{2\alpha} - I_{\H_1}]A^{*}Bh_{2\alpha}(B^{*}B)
	\\
	B^{*}Ah_{2\alpha}(A^{*}A) & [(I_{\H_1} + B^{*}B)^{2\alpha} - I_{\H_1}] & B^{*}Ah_{2\alpha}(A^{*}A)A^{*}Bh_{2\alpha}(B^{*}B)
	\\
	B^{*}Ah_{2\alpha}(A^{*}A) &  [(I_{\H_1} + B^{*}B)^{2\alpha} - I_{\H_1}] & B^{*}Ah_{2\alpha}(A^{*}A)A^{*}Bh_{2\alpha}(B^{*}B)
	\end{pmatrix}
	\end{align*}
	We finally obtain
	\begin{align*}
	&\trace[[(I_{\H_2} + AA^{*})^{2\alpha}(I_{\H_2} + BB^{*})^{2\alpha}]^{1/2} - I_{\H_2}]
	\\
	&= 
	\trace\left[\left[I_{\H_1^3} 
	+
	\begin{pmatrix}
	C_{11} & C_{12} & C_{13}
	\\
	C_{21} & C_{22} & C_{23}
	\\
	C_{21} & C_{22} & C_{23}
	\end{pmatrix}
	\right]^{1/2} - I_{\H_1^3}\right],
	\end{align*}
	where
	\begin{align*}
	&C_{11} = [(I_{\H_1} + A^{*}A)^{2\alpha} - I_{\H_1}], \; \; C_{12} = A^{*} Bh_{2\alpha}(B^{*}B),
	\\
	&C_{13} =[(I_{\H_1} + A^{*}A)^{2\alpha} - I_{\H_1}]A^{*}Bh_{2\alpha}(B^{*}B),
	\\
	&C_{21} =B^{*}Ah_{2\alpha}(A^{*}A), \; \; C_{22} =[(I_{\H_1} + B^{*}B)^{2\alpha} - I_{\H_1}],
	\\
	&C_{23} = B^{*}Ah_{2\alpha}(A^{*}A)A^{*}Bh_{2\alpha}(B^{*}B).
	\end{align*}
	This completes the proof. \qed
\end{proof}

\begin{proof}
	[\textbf{of Corollary \ref{corollary:dproHS-H1H2-gamma}}]	
	By definition,
	\begin{align*}
	&\alpha^2(d^{\alpha}_{\proHS}[(AA^{*}+\gamma I_{\H_2}), (BB^{*}+\gamma I_{\H_2})])^2
	\\
	& = \alpha^2\gamma^{2\alpha}(d^{\alpha}_{\proHS}[(AA^{*}/\gamma + I_{\H_2}), (BB^{*}/\gamma + I_{\H_2})])^2.
	\end{align*}
	The desired result then follows from Proposition \ref{proposition:dproHS-H1H2}. \qed
\end{proof}

\begin{lemma}
	\label{lemma:h-alpha-gamma-0}
Let $\H_1$ be a separable Hilbert space and $E:\H_1 \mapto \H_1$ be a positive compact operator on $\H_1$. Let $h_{\alpha}(E)$ be as defined in Eq.(\ref{equation:h-alpha}).
Let $\alpha \geq 1/2$ be fixed. Then in the operator norm on $\H_1$,
\begin{align}
&\lim_{\gamma \approach 0^{+}}\gamma^{2\alpha-1}h_{2\alpha}\left(\frac{E}{\gamma}\right) = E^{2\alpha-1},
\end{align}
where the right hand side is a positive bounded operator on $\H_1$. If $\alpha > 1/2$, 
then $E^{2\alpha-1}$ is compact.
	\end{lemma}
\begin{proof}
For any $\gamma > 0$, by definition of $h_{\alpha}(E)$ as in Eq.(\ref{equation:h-alpha}), we have
\begin{align*}
&\gamma^{2\alpha-1}h_{2\alpha}(\frac{E}{\gamma}) =\gamma^{2\alpha-1}\sum_{k=1}^{N_E}\frac{[1+(\lambda_k(E)/\gamma)]^{2\alpha}-1}{(\lambda_k(E)/\gamma)}\phi_k(E) \otimes \phi_k(E)
\\
& = \sum_{k=1}^{N_E}\frac{(\gamma + \lambda_k(E))^{2\alpha} - \gamma^{2\alpha}}{\lambda_k(E)}\phi_k(E) \otimes \phi_k(E).
\end{align*}
Thus for $\alpha \geq 1/2$, we have in the operator norm on $\H_1$,
\begin{align*}
\lim_{\gamma \approach 0}\gamma^{2\alpha-1}h_{2\alpha}(\frac{E}{\gamma}) = \sum_{k=1}^{N_E}\lambda_k^{2\alpha-1}(E)\phi_k(E) \otimes \phi_k(E) = E^{2\alpha-1}.
\end{align*}
The right hand side is clearly a positive bounded operator for $\alpha \geq 1/2$ and compact when $\alpha >1/2$.
\qed
\end{proof}

The following is straightforward to verify.
\begin{lemma}
\label{lemma:eigen-block-matrix}
Let $A,B$ be real $n \times n$ matrix.
Consider the $3n \times 3n$ block matrix
\begin{align}
C = \begin{pmatrix}
0 & 0 & A
\\
0& 0 & 0
\\
0 & 0 & B
\end{pmatrix}.
\end{align}
Then the nonzero eigenvalues of $C$ and $B$ are the same.
\end{lemma}

\begin{proof}
	[\textbf{of Corollary \ref{corollary:dproHS-H1H2-gamma-0}}]
By Lemma \ref{lemma:h-alpha-gamma-0}, we have
\begin{align*}
&\lim_{\gamma \approach 0}\gamma^{2\alpha-1}h_{2\alpha}(\frac{A^{*}A}{\gamma}) = (A^{*}A)^{2\alpha-1}, \lim_{\gamma \approach 0}\gamma^{2\alpha-1}h_{2\alpha}(\frac{B^{*}B}{\gamma}) = (B^{*}B)^{2\alpha-1}.
\end{align*}
Since $A^{*}B, B^{*}A$ are trace class operators on $\H_1$ by assumption, using the expression in Corollary \ref{corollary:dproHS-H1H2-gamma}, we have
\begin{align*}
&\lim_{\gamma \approach 0}
\alpha^2(d^{\alpha}_{\proHS}[(AA^{*} + \gamma I_{\H_2}), (BB^{*} + \gamma I_{\H_2})])^2
\nonumber
\\
& = \trace[(A^{*}A)^{2\alpha}] + \trace[(B^{*}B)^{2\alpha}]
- 2\trace
\begin{pmatrix}
0 & 0 & (A^{*}A)^{2\alpha}A^{*}B(B^{*}B)^{2\alpha-1}
\\
0 & 0 & 0
\\
0 & 0 & B^{*}A(A^{*}A)^{2\alpha-1}A^{*}B(B^{*B})^{2\alpha-1}
\end{pmatrix}^{1/2}
\\
& = \trace[(A^{*}A)^{2\alpha} + (B^{*}B)^{2\alpha} - 2(B^{*}A(A^{*}A)^{2\alpha-1}A^{*}B(B^{*}B)^{2\alpha-1})^{1/2}],
\end{align*}
where the last equality follows from Lemma \ref{lemma:eigen-block-matrix}.

For $\alpha = 1/2$, we can also proceed directly as follows. By Corollary 2 in \cite{Minh:LogDetIII2018},
	\begin{align}
	\label{equation:h1}
	Ah_1(A^{*}A) = A, \;\; h_1(A^{*}A)A^{*} = A^{*}.
	\end{align}
	It follows that
	\begin{align*}
	&C_{11} = \frac{1}{\gamma}A^{*}A, \;\; C_{12} = \frac{1}{\gamma}A^{*}Bh_1(B^{*}B/\gamma) = \frac{1}{\gamma} A^{*}B,\;\;
	\\
	&C_{13} = \frac{1}{\gamma}(A^{*}A/\gamma) A^{*}Bh_1(B^{*}B/\gamma) = \frac{1}{\gamma^2}A^{*}AA^{*}B,
	\\
	&C_{21} = \frac{1}{\gamma}B^{*}Ah_1(A^{*}A/\gamma) = \frac{1}{\gamma}B^{*}A, \;\; C_{22} = \frac{1}{\gamma}B^{*}B,\;\;
	\\
	&C_{23} = \frac{1}{\gamma^2}B^{*}Ah_1(A^{*}A/\gamma)A^{*}Bh_1(B^{*}B/\gamma) = \frac{1}{\gamma^2}B^{*}AA^{*}B.
	\end{align*}
	Using the expression in Corollary \ref{corollary:dproHS-H1H2-gamma}, we have
	\begin{align*}
	&\lim_{\gamma \approach 0}
	\frac{1}{4}(d^{1/2}_{\proHS}[(AA^{*} + \gamma I_{\H_2}), (BB^{*} + \gamma I_{\H_2})])^2
	\nonumber
	\\
	& = \trace[(A^{*}A] + \trace[(B^{*}B]
	- 2\trace
	\begin{pmatrix}
	0 & 0 & A^{*}AA^{*}B
	\\
	0 & 0 & 0
	\\
	0 & 0 & B^{*}AA^{*}B
	\end{pmatrix}^{1/2}
	\\
	& = \trace[(A^{*}A) + (B^{*}B) - 2(B^{*}AA^{*}B)^{1/2}],
	\end{align*}
	where the last equality follows from Lemma \ref{lemma:eigen-block-matrix}.
	\qed
\end{proof}

\subsection{Proofs for the distances between Gaussian measures}

\begin{lemma}
	\label{lemma:inequality-quadratic}
	For any real numbers $a,b,c,d$,
	\begin{align}
	\sqrt{(a+b)^2 + (c+d)^2} \leq \sqrt{a^2 + c^2} + \sqrt{b^2 + d^2}.
	\end{align}
\end{lemma}
\begin{proof}
	We have
	\begin{align*}
	&\sqrt{(a+b)^2 + (c+d)^2} \leq \sqrt{a^2 + c^2} + \sqrt{b^2 + d^2} 
	\\
	&\equivalent 
	(a+b)^2 + (c+d)^2 \leq a^2 + c^2 + b^2 + d^2 + 2\sqrt{(a^2 + c^2)(b^2 + d^2)}
	\\
	& \equivalent (ab +cd) \leq \sqrt{(a^2 + c^2)(b^2 + d^2)}
	\end{align*}
	which is obvious if $(ab+cd) \leq 0$. If $(ab + cd) > 0$, then this is equivalent to 
	$(ad-bc)^2 \geq0$, which is obvious. \qed
\end{proof}

\begin{proof}
	[\textbf{of Lemma \ref{lemma:metric-quadratic-Gaussian}}]
	It suffices to show the triangle inequality for $D$. We have
	\begin{align*}
	&D(\Ncal(m_1, C_1), \Ncal(m_3,C_3)) = \sqrt{d_{\mean}^2(m_1,m_3) + d_{\cov}^2(C_1, C_3)}
	\\
	& \leq \sqrt{[d_{\mean}(m_1,m_2) + d_{\mean}(m_2,m_3)]^2 + [d_{\cov}(C_1,C_2) + d_{\cov}(C_2,C_3)]^2 }
	\\
	& \leq \sqrt{d_{\mean}^2(m_1,m_2) + d_{\cov}^2(C_1, C_2)} + \sqrt{d_{\mean}^2(m_2,m_3) + d_{\cov}^2(C_2, C_3)}\
	\\
	& = D(\Ncal(m_1, C_1)||\Ncal(m_2,C_2)) + D(\Ncal(m_2, C_2)||\Ncal(m_3,C_3)),
	\end{align*}
	where the last inequality follows from Lemma \ref{lemma:inequality-quadratic}. \qed
\end{proof}

\begin{proof}
	[\textbf{of Lemma \ref{lemma:trace-power-1}}]
	Let $\{\lambda_k\}_{k=1}^{\infty}$ be the eigenvalues of $A$, then by the positive trace class assumption, $\lambda_k \geq 0$, $\lim_{k\approach \infty} \lambda_k = 0$, and
	$\sum_{k=1}^{\infty} \lambda_k < \infty$. Thus with $\alpha \geq 1/2$, so that $2\alpha \geq 1$, we also have
	$\sum_{k=1}^{\infty}\lambda_k^{2\alpha} < \infty$. It follows that $A^{2\alpha} \in \Sym^{+}(\H) \cap \Tr(\H)$ and
	$A^{\alpha} \in \Sym^{+}(\H) \cap \HS(\H)$. \qed
\end{proof}

\begin{proof}
	[\textbf{of Lemma \ref{lemma:trace-power-2}}]
	By Lemma \ref{lemma:trace-power-2}, we have $A^{\alpha}, B^{\alpha} \in \Sym^{+}(\H) \cap \HS(\H)$, 
	so that $A^{\alpha}B^{\alpha} \in \Tr(\H)$. By definition of the trace norm, we then have
	\begin{align*}
	\trace[(A^{\alpha}B^{2\alpha}A^{\alpha})]^{1/2} = \trace[(A^{\alpha}B^{\alpha})(A^{\alpha}B^{\alpha})^{*}]^{1/2} = ||(A^{\alpha}B^{\alpha})||_{\tr} < \infty.
	\end{align*}
	It thus follows that $(A^{\alpha}B^{2\alpha}A^{\alpha})^{1/2} \in \Sym^{+}(\H) \cap \Tr(\H)$. \qed
	\end{proof}

\begin{proof}
	[\textbf{of Theorem \ref{theorem:metric-Gaussian-Hilbert-zeromean}}]
	 We observe that for a fixed $\gamma > 0$,
	$D^{\alpha}_{\proHS}[\Ncal(0,C_1)||\Ncal(0,C_2)] = 0 \equivalent d^{\alpha}_{\proHS}[(C_1 + \gamma I), (C_2 + \gamma I)] = 0 \equivalent
	C_1 + \gamma I = C_2 + \gamma I \equivalent C_1 = C_2$. The theorem then follows from the metric properties of $d^{\alpha}_{\proHS}$.
	\qed
	\end{proof}

\begin{proof}
	[\textbf{of Theorem \ref{theorem:metric-Gaussian-Rn}}]
	This follows from Theorem 
	\ref{theorem:metric-family-finite} and Lemma
	\ref{lemma:metric-quadratic-Gaussian}.
	\qed
\end{proof}

\begin{proof}
	[\textbf{of Theorem \ref{theorem:metric-Gaussian-Hilbert}}]
	This follows from Theorem 
	\ref{theorem:metric-Gaussian-Hilbert-zeromean} and Lemma
	\ref{lemma:metric-quadratic-Gaussian}.
	\qed
\end{proof}

\begin{proof}
	[\textbf{of Theorem \ref{theorem:distance-Gaussian-RKHS-alpha}}]
	By definition of $\mu_{\Phi(\Xbf)}$, $\mu_{\Phi(\Ybf)}$,
	\begin{align*}
	&||\mu_{\Phi(\Xbf)} - \mu_{\Phi(\Ybf)}||^2_{\H_K} = \frac{1}{m^2}||\Phi(\Xbf)\1_m - \Phi(\Ybf)\1_m||^2_{\H_K} 
	\\
	&= \frac{1}{m^2}
	\left\la[\sum_{i=1}^m\Phi(x_i) - \sum_{j=1}^m\Phi(y_j)], [\sum_{i=1}^m\Phi(x_i) - \sum_{j=1}^m\Phi(y_j)\ra_{\H_K}]\right\ra_{\H_K}
	\\
	& = \frac{1}{m^2}\left[\sum_{i,j=1}^m[K(x_i,x_j) + K(y_i,y_j) - 2K(x_i,y_j)]\right]
	\\
	&= \1_m^T(K[\Xbf] + K[\Ybf] - 2K[\Xbf, \Ybf])\1_m.
	\end{align*}
	The desired result is then obtained by combining Theorem \ref{theorem:metric-Gaussian-Hilbert}
	with Theorem \ref{theorem:proHS-RKHS} and the above expression. \qed
\end{proof}

\begin{proof}
	[\textbf{of Theorem \ref{theorem:distance-Gaussian-RKHS-alpha-more-than1/2}}]
	This is similar to the proof of Theorem \ref{theorem:distance-Gaussian-RKHS-alpha}, except that we invoke Theorem 
	\ref{theorem:metric-Gaussian-Hilbert} for the case $\alpha \geq 1/2$.
	\qed
\end{proof}
\end{document}